\documentclass[11pt]{amsart}
\usepackage{extarrows}
\usepackage{latexsym}
\usepackage{amssymb,amsmath,mathabx, amscd}
\usepackage[pdftex]{graphicx}
\usepackage{enumerate}
\usepackage{tikz}
\usepackage[margin=1in]{geometry}
\usepackage{listings}
\usepackage{courier}
\usepackage{tikz-cd, wrapfig}
\usetikzlibrary{decorations.pathreplacing}
\usepackage{color}
\usepackage{MnSymbol}
\usepackage{hyperref}
\usepackage{mathrsfs, colonequals}

\renewcommand{\aa}{\mathbb{A}}

\newcommand{\pp}{\mathbb{P}}

\renewcommand{\L}{\mathcal{L}}

\renewcommand{\O}{\mathcal{O}}

\newcommand{\N}{\mathcal{N}}

\newcommand{\sC}{\mathscr{C}}

\newcommand{\sE}{\mathscr{E}}
\newcommand{\sF}{\mathscr{F}}

\newcommand{\sP}{\mathscr{P}}

\newcommand{\chara}{\operatorname{char}}

\newcommand{\codim}{\operatorname{codim}}

\newcommand{\Hom}{\operatorname{Hom}}

\newcommand{\rk}{\operatorname{rk}}

\newcommand{\del}{\partial}
\renewcommand{\bar}{\overline}

\newcommand{\leqpar}{\underset{{\scriptscriptstyle (}-{\scriptscriptstyle )}}{<}}

\newcommand{\defi}[1]{\textsf{#1}}

\newtheorem{thm}{Theorem}[section]
\newtheorem{ithm}{Theorem}

\newtheorem{lem}[thm]{Lemma}
\newtheorem{conj}[thm]{Conjecture}
\newtheorem{prop}[thm]{Proposition}
\newtheorem{cor}[thm]{Corollary}

\theoremstyle{definition}
\newtheorem{defin}[thm]{Definition}
\newtheorem{example}[thm]{Example}
\newtheorem{assumption}[thm]{Assumption}

\theoremstyle{remark}
\newtheorem{rem}{Remark}

\title{Stability of Normal Bundles of Space Curves}
\author{Izzet Coskun}
\address{Department of Mathematics, Statistics, and CS \\
University of Illinois at Chicago, Chicago IL 60607}
\email{coskun@math.uic.edu}
\author{Eric Larson}
\address{Department of Mathematics, University of Washington \\
Padelford Hall (PDL), Seattle, WA 98195}
\email{elarson3@gmail.com}
\author{Isabel Vogt}
\address{Department of Mathematics, University of Washington \\
Padelford Hall (PDL), Seattle, WA 98195}
\email{ivogt.math@gmail.com}

\thanks{During the preparation of this article, I.C.\ was supported
by NSF FRG grant DMS-1664296, and E.L.\ and I.V.\ were supported by
NSF MSPRF grants DMS-1802908 and DMS-1902743 respectively.}

\begin{document}
\maketitle

\begin{abstract}
In this paper, we prove that the normal bundle of a general Brill-Noether space curve of degree $d$ and genus $g \geq 2$ is stable if and only if $(d,g) \not\in \{ (5,2), (6,4) \}$.  When $g\leq1$ and the characteristic of the ground field is zero, it is classical that the normal bundle is strictly semistable.  We show that this still holds in positive characteristic except when the characteristic is $2$, the genus is \(0\) and the degree is even.\end{abstract}

\section{Introduction}
Let $C$ be a smooth connected curve defined over an algebraically closed field $k$ (of arbitrary characteristic).  The normal bundle $N_{C/\pp^r}$ of a smooth curve controls the deformations of the curve in $\pp^r$ and plays a crucial role in many problems of geometry, arithmetic and commutative algebra. In this paper, we show that the normal bundle of a general  Brill-Noether space curve of degree $d$ and genus $g$ is stable if and only if $g \geq 2$ and $(d,g)\not\in \{ (5,2), (6,4) \}$.

Let $E$ be a vector bundle on a smooth curve $C$.  Let the \defi{slope $\mu(E)$} be
\[\mu(E) \colonequals \frac{\deg(E)}{\rk(E)}.\]
Then $E$ is called  \defi{(semi)stable} if every proper subbundle \(F\)
(which is always assumed to be saturated in this paper)
of smaller rank satisfies 
\[\mu(F) \leqpar \mu(E).\] 
The  bundle is called \defi{unstable} if it is not semistable and \defi{strictly semistable} if it is semistable but not stable. 

By the Brill-Noether Theorem (see \cite{kleimanlaksov, griffithsharris, gieseker, acgh, oss, jp, clt}), a general curve of genus $g$ admits a nondegenerate, degree $d$ map to $\pp^r$ if and only if the Brill-Noether number $\rho(g,r,d)$ satisfies $$\rho(g,r,d) : = g - (r+1)(g-d+r) \geq 0.$$  When $r \geq 3$, there is a unique component of the Hilbert scheme that dominates the moduli space $\overline{M}_g$ and whose general member parameterizes a smooth, nondegenerate curve of degree $d$ and genus $g$ in $\pp^r$. We call a member of this component a \defi{Brill--Noether curve}. When $r=3$, we call such a curve a  \defi{Brill-Noether space curve}.  With this terminology, our main theorem is the following.

\begin{ithm}\label{thm:main}
Let $C \subseteq \pp^3$ be a general Brill--Noether space curve of degree $d$ and genus $g$
over an algebraically closed field $k$.
\begin{enumerate}
\item $N_C$ is stable if and only if $g\geq 2$ and $(d,g) \not\in \{ (5,2), (6,4) \}$.
\item $N_C$ is strictly semistable if and only if $g < 2$ and one of the following holds: $\operatorname{char}(k) \neq 2$, $g = 1$, or $d$~is odd.
\item $N_C$ is unstable if and only if $(d, g) \in \{(5, 2), (6, 4)\}$,
or all of the following hold: $\operatorname{char}(k)=2$, $g = 0$, and $d$ is even.
\end{enumerate}
\end{ithm}

The normal bundles of curves in projective space have been studied by many authors (for example, see \cite{aly, ballicoellia, coskunriedl, einlazarsfeld, ellia, ellingsrudhirschowitz, ellingsrudlaksov, newstead, ran, sacchiero, sacchiero2, sacchiero3}). Our results complete and unify these results for Brill-Noether space curves. 

If $(d, g) \in \{(5, 2), (6, 4)\}$, then $C$ lies on a unique quadric $Q$ and $N_{C/Q} \subset N_C$ gives a destabilizing subbundle. We will describe the geometry in these two cases more explicitly in \S \ref{sec-unstable}.

Every bundle on $\pp^1$ splits as a direct sum of line bundles. Hence, the normal bundle of a smooth rational curve can be written as $N_{C} = \bigoplus_{i=1}^{r-1}\O(a_i)$ for some integers $a_1, \dots, a_{r-1}$ with $$\sum_{i=1}^{r-1} a_i = (r+1)d -2.$$ If $C$ is a general rational curve of degree at least $d \geq r$ in $\pp^r$, and the characteristic of the ground field is not $2$, then $N_{C/ \pp^r}$ splits as equally as possible, i.e. $|a_i - a_j| \leq 1$ (see \cite{sacchiero, ran, coskunriedl, aly}). Hence, $N_{C/ \pp^r}$ is strictly semistable when $r-1$ divides $2d -2$ and is unstable otherwise. When $r=3$ and $\text{char}(k)\neq 2$, since the quantity $2d-2$ is always even, the normal bundle of a general rational curve of degree $d \geq 3$ is strictly semistable.  If the characteristic is $2$, we show in Lemma \ref{lem:char2_obs} that all $a_i \equiv d$ mod $2$; this obstructs semistability for rational curves with $d$ even.

Similarly, normal bundles of genus one curves have been studied extensively (see \cite{einlazarsfeld, ellingsrudhirschowitz, ellingsrudlaksov}). By \cite{ellingsrudhirschowitz}, the normal bundle of a general nondegenerate genus one space curve is semistable. On the other hand, on a genus one curve, there are no stable rank $2$ bundles of degree $4d$. Hence, the normal bundle of a general genus one space curve of degree $d\geq 4$ is strictly semistable. Our techniques will provide short arguments reproving the $g=0$ and $1$ cases.

In higher genus, the previously known results were more sporadic.  The stability of the normal bundle was proved  for $(d, g)=(6, 2)$ by Sacchiero \cite{sacchiero3}, for $(d, g)=(9, 9)$ by Newstead \cite{newstead}, for $(d, g)=(6, 3)$ by Ellia \cite{ellia}, and for $(d, g)=(7, 5)$ by Ballico and Ellia \cite{ballicoellia}. Many of these cases will be important for our inductive arguments. For completeness, we will reprove these cases using our techniques or briefly recall
the arguments.  More generally, in \cite{ellingsrudhirschowitz}, Ellingsrud and Hirschowitz announced a proof of stability of normal bundles in an asymptotic range of degrees and genera; however, their results do not cover many of the most challenging cases of small degree.

We prove Theorem \ref{thm:main} by specialization.  We use three basic specializations: (1) we specialize to a curve of degree $(d-1, g)$ union a $1$-secant line; (2) we specialize to a curve of degree $(d-1, g-1)$ union a $2$-secant line; and finally (3) we specialize to a curve of degree $(d-2, g-3)$ union a $4$-secant conic.  These degenerations reduce Theorem \ref{thm:main} to a finite set of base cases.
The most challenging part of the paper is to verify these base cases.

We expect our techniques and results to generalize to $\pp^r$ for $r\geq 3$ and hopefully settle the following conjecture.

\begin{conj} \label{conj:larger}
The normal bundle of a general Brill-Noether curve of genus at least $2$ in $\pp^r$ is stable except for finitely many triples $(d,g, r)$. 
\end{conj}

Conjecture \ref{conj:larger} is closely related to several conjectures in the literature. For example, Aprodu, Farkas and Ortega have conjectured that the normal bundle of a general canonical curve of $g \geq 7$ is stable \cite[Conjecture 0.4]{AFO} (see also \cite{Bruns}).


\subsection*{Organization of the paper} In  \S \ref{sec-prelim}, we will recall basic facts about normal bundles on nodal curves and elementary modifications. In  \S \ref{sec-unstable}, we will elaborate on the two cases $(d,g) \in \{(5, 2), (6, 4)\}$ as well as the obstruction to stability for rational curves in characteristic $2$. In  \S \ref{sec:stab_degen1}, we will introduce several basic degenerations to reduce the theorem to a small set of initial cases. For the rest of the paper, we will analyze these initial cases.

\subsection*{Acknowledgments} We would like to thank Atanas Atanasov, Lawrence Ein, Gavril Farkas, Joe Harris, Eric Riedl, Ravi Vakil, and David Yang for invaluable conversations on normal bundles of curves.
We would also like to thank the anonymous referee for many useful suggestions.

\section{Preliminaries}\label{sec-prelim}

In this section, we collect basic facts on normal bundles of curves,
stability of vector bundles, elementary modifications, and on certain
reducible Brill--Noether curves.
For more details, we refer the reader to \cite{aly, rbn, rbn2}.
When necessary, we provide a characteristic-independent proof here.

\subsection{The normal bundle of a space curve}
Let $C \subset \pp^r$ be a smooth Brill--Noether curve of degree $d$ and genus $g$.  The normal bundle $N_C$ is a rank $r-1$ vector bundle that is presented as a quotient
\[0 \to T_C \to T_{\pp^r}|_C \to N_C \to 0,\]
of the restricted tangent bundle of $\pp^r$ by the tangent bundle of $C$.  The restricted tangent bundle is itself naturally a quotient 
 in the Euler exact sequence
\begin{equation}\label{euler_seq} 0 \to \O_C \to \O_C(1)^{\oplus (r+1)} \to T_{\pp^r}|_C \to 0.\end{equation}
From this we see that $\deg(N_C) = (r+1)d  + 2g - 2$.
Specializing to $r=3$, we have that
\[\mu(N_C) = 2d + g -1,\]
and therefore $N_C$ is stable if and only if all line subbundles $L \subseteq N_C$ have slope at most $2d + g -2$.

If \(S\) is a surface in \(\pp^3\) containing \(C\) that is smooth at the generic point of \(C\), then we have an associated \defi{normal bundle exact sequence}
\begin{equation}\label{eq:normal_bundle_exact_seq}
0 \to N_{C/S} \to N_C \to N_{S}|_C \to 0.
\end{equation}
By adjunction, the bundle \(N_{C/S}\) is isomorphic to \(\O_S(C)|_C\).  A particularly simple case is when \(C\) is the complete intersection of two (smooth) surfaces \(S_1\) and \(S_2\) of degrees \(d_1\) and \(d_2\) in \(\pp^3\).  In this case the natural map
\[N_{C/S_1} \oplus N_{C/S_2} \to N_C\]
is an isomorphism, and, combining this with the adjunction isomorphism, we have \(N_{C} \simeq \O_C(d_1) \oplus \O_C(d_2)\).  Such a bundle is never stable, and is semistable if and only if \(d_1 = d_2\).  Relevant examples for us are lines (the normal bundle is isomorphic to \(\O_{\pp^1}(1)^{\oplus 2}\)), conics (the normal bundle is isomorphic to \(\O_{\pp^1}(2) \oplus \O_{\pp^1}(4)\)) and elliptic quartics $E$ (the normal bundle is isomorphic to $\O_E(2) \oplus \O_E(2)$).

\subsection{Stability of vector bundles on nodal curves}

In the course of our inductive argument, we will specialize a smooth Brill--Noether curve to a reducible nodal curve.
In this section, we generalize the definition of stability of vector bundles to allow $C$ to be a connected nodal curve.  We will write 
\[\nu \colon \tilde{C} \to C\] 
for the normalization of $C$.  For any node $p$ of $C$, write $\tilde{p}_1$ and $\tilde{p}_2$ for the two points of $\tilde{C}$ over $p$.

Given a vector bundle $E$ on $C$, the fibers of the pullback $\nu^*E$ to $\tilde{C}$ over $\tilde{p}_1$ and $\tilde{p}_2$ are naturally identified.  Given a subbundle $F \subseteq \nu^*E$, it therefore makes sense to compare $F|_{\tilde{p}_1}$ and $F|_{\tilde{p}_2}$ inside $\nu^*E|_{\tilde{p}_1} \simeq \nu^*E|_{\tilde{p}_2}$.

\begin{defin}
Let $E$ be a vector bundle on a connected nodal curve $C$.  
For a subbundle $F \subset \nu^*E$, define the \defi{adjusted slope $\mu^{\text{adj}}_C$} by 
\[\mu^{\text{adj}}_C(F) \colonequals \mu(F) - \frac{1}{\rk{F}} \sum_{p \in C_{\text{sing}}} \codim_{F} \left(F|_{\tilde{p}_1}\cap F|_{\tilde{p}_2} \right),\]
where $\codim_F \left(F|_{\tilde{p}_1}\cap F|_{\tilde{p}_2} \right)$ refers to the codimension
of the intersection in either $F|_{\tilde{p}_1}$ or $F|_{\tilde{p}_2}$ (which are equal
since $\dim F|_{\tilde{p}_1} = \dim F|_{\tilde{p}_2}$). 
When the curve $C$ is unambiguous, we will omit it from our notation and write simply $\mu^{\text{adj}}(F)$.
Note that if $F$ is pulled back from $C$, then $\mu^{\text{adj}}_C(F) = \mu(F)$.
We say that $E$ is \defi{(semi)stable} if for all subbundles $F \subset \nu^*E$,
\[\mu^{\text{adj}}(F) \leqpar \mu(\nu^*E) = \mu(E). \]
\end{defin}

With this definition, stability is an open condition in families of connected nodal curves.
To show this, we will need the following lemma.

\begin{lem} \label{lem:contract} Let $\beta \colon C' \to C$ be a map obtained by contracting a $1$- or $2$- secant $\pp^1$:
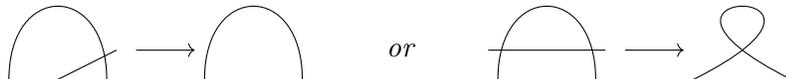
\begin{center}
\begin{tikzpicture}[scale=1.3]
\draw (0, 0) .. controls (0, 1) and (1, 1) .. (1, 0);
\draw (0.5, 0) -- (1.1, 0.3);
\draw (2, 0) .. controls (2, 1) and (3, 1) .. (3, 0);
\draw[->] (1.3, 0.3) -- (1.9, 0.3);
\draw (4, 0.3) node{or};
\draw (5, 0) .. controls (5, 1) and (6, 1) .. (6, 0);
\draw (4.9, 0.3) -- (6.1, 0.3);
\draw (7, 0) .. controls (9, 1) and (6, 1) .. (8, 0);
\draw[->] (6.3, 0.3) -- (6.9, 0.3);
\end{tikzpicture}
\end{center}
If $E$ is a (semi)stable vector bundle on $C$, then $\beta^* E$ is also (semi)stable.
\end{lem}
\begin{proof}
Write $\nu \colon \tilde{C} \to C$ and $\nu' \colon \tilde{C'} \to C'$
for the normalization maps.

First consider the $1$-secant case. Write $x$ for the point of attachment
(so that $C' = C \cup_x \pp^1$).  Let $E$ be a (semi)stable vector bundle on $C$, and let $F \subset \nu'^* \beta^* E$ be any subbundle.
Since $\nu'^*\beta^*E|_{\pp^1}$ is trivial, we have
\begin{equation} \label{muf}
\mu(F|_{\pp^1}) \leq 0.
\end{equation}
Since the ordinary slope is additive on components, 
and \(x\) is the only point in \(C'_{\text{sing}}\) that is not also in \(C_{\text{sing}}\), 
the definition of adjusted slope and
\eqref{muf} imply
\[\mu^{\text{adj}}_{C'}(F) = \mu^{\text{adj}}_C (F|_{\tilde{C}}) + \mu(F|_{\pp^1}) - \frac{\codim_F \left(F|_{\tilde{x}_1}\cap F|_{\tilde{x}_2} \right)}{\rk F} \leq \mu^{\text{adj}}_C (F|_{\tilde{C}}) \leqpar \mu(E),\]
hence $\beta^*E$ is (semi)stable.

Similarly in the $2$-secant case, write $C' = C'' \cup_{\{x, y\}} \pp^1$.
Denote by $\tilde{x}_1$ and $\tilde{y}_1$ (respectively $\tilde{x}_2$ and~$\tilde{y}_2$)
the corresponding points on $\pp^1$ (respectively $C''$).  Let $F \subset \nu'^*\beta^*E$.
Since $\nu'^*\beta^*E|_{\pp^1}$ is trivial, we can identify the fiber of $E$ at $\tilde{x}_1$ with the fiber of $E$ at $\tilde{y}_1$, and we have
\[\mu(F|_{\pp^1}) \leq -\frac{1}{\rk F} \cdot  \codim_F \left(F|_{\tilde{x}_1}\cap F|_{\tilde{y}_1} \right).\]
As in the \(1\)-secant case, noting that the only difference between \(C'_{\text{sing}}\) and \(C''_{\text{sing}}\) are the points \(\{x,y\}\), for any subbundle $F \subset \nu'^* \beta^* E$, we have
\begin{align*}
\mu^{\text{adj}}_{C'}(F) &= \mu^{\text{adj}}_{C''}(F|_{\tilde{C}}) + \mu(F|_{\pp^1}) - \frac{1}{\rk F} \cdot \bigg(\codim_F \left(F|_{\tilde{x}_1}\cap F|_{\tilde{x}_2} \right) + \codim_F \left(F|_{\tilde{y}_1}\cap F|_{\tilde{y}_2} \right) \bigg) \\
&\leq \mu^{\text{adj}}_{C''}(F|_{\tilde{C}}) - \frac{1}{\rk F} \cdot \bigg(\codim_F \left(F|_{\tilde{x}_1}\cap F|_{\tilde{y}_1}\right) + \codim_F \left(F|_{\tilde{x}_1}\cap F|_{\tilde{x}_2} \right) + \codim_F \left(F|_{\tilde{y}_1}\cap F|_{\tilde{y}_2} \right) \bigg) \\
\intertext{Twice applying the ``triangle inequality'' $\codim (X \cap Y) + \codim (Y \cap Z) \geq \codim (X \cap Z)$,}
&\leq \mu^{\text{adj}}_{C''}(F|_{\tilde{C}}) - \frac{1}{\rk F} \cdot \codim_F \left(F|_{\tilde{x}_2}\cap F|_{\tilde{y}_2} \right) \\
&= \mu^{\text{adj}}_C(F|_{\tilde{C}}) \\
&\leqpar \mu(E). \qedhere
\end{align*}
\end{proof}

\begin{prop} \label{prop:stab-open}
Let $\sC \to \Delta$ be a family of connected nodal curves over
the spectrum of a discrete valuation ring,
and $\sE$ be a vector bundle on $\sC$.
\begin{enumerate}
\item \label{stab-open}
If the special fiber $\sE_0 = \sE|_0$ is (semi)stable,
then the general fiber $\sE^* = \sE|_{\Delta^*}$ is also (semi)stable.
\item \label{ssemi-glob}
If $\sC \to \Delta$ is smooth, and $\sE_0$ is semistable, 
then any subbundle $\sF^* \subset \sE^*$
with $\mu(\sF^*) = \mu(\sE^*)$ extends to a subbundle $\sF \subset \sE$.
\end{enumerate}
\end{prop}
\begin{proof}
Write $\nu \colon \tilde{\sC} \to \sC$ for the normalization.

For part~(\ref{stab-open}), after possibly making a base change,
let $\sF^* \subset \nu^* \sE^*$ be a subbundle
with $\mu(\sF^*)$ maximal.
Since $\mu$ is constant in flat families and $\codim(X \cap Y)$ is lower semicontinuous,
$\mu^{\text{adj}}$ is upper semicontinuous in flat families.
Therefore, if $\sF^*$ extends to a subbundle $\sF \subset \nu^* \sE$, then
\begin{equation} \label{stab-open-easy}
\mu^{\text{adj}}(\sF^*) \leq \mu^{\text{adj}}(\sF_0) \leqpar \mu(\sE_0) = \mu(\sE^*).
\end{equation}
Otherwise,
we make a blowup $\tilde{\beta} \colon \tilde{\sC'} \to \tilde{\sC}$
in order to extend $\sF^* \subset \nu^* \sE^*$ to a subbundle $\sF \subset \tilde{\beta}^* \nu^* \sE$.
By semistable reduction,
we may ensure that the central fiber remains reduced.
By gluing along sections identified under $\nu$, the blowup $\tilde{\beta}$
induces a map $\beta \colon \sC' \to \sC$, which is an isomorphism away from the central fiber, and on the central fiber consists of replacing nodes by $1$- and $2$-secant $\pp^1$'s.
Applying Lemma~\ref{lem:contract}, $\beta^* \sE_0$ is (semi)stable.
Therefore \eqref{stab-open-easy} holds for $\beta^* \sE$.

For part~(\ref{ssemi-glob}), we imitate the above argument to extend
$\sF^*$ to a subbundle of $\beta^* \sE$.
Since $\sC \to \Delta$ is smooth,
$\beta$ can be obtained by iteratively contracting $1$-secant $\pp^1$s.
Since $\mu(\sF^*) = \mu(\sE^*)$ and $\beta^* \sE_0$ is semistable,
we must in particular have equality in equation \eqref{muf} from the proof of Lemma \ref{lem:contract} for every such contraction;
thus, $\sF$ is trivial along every exceptional divisor of $\beta$.
In particular, $\sF^*$ already extends to a subbundle $\sF \subset \sE$
without blowing up.
\end{proof}

\subsection{Elementary modifications of vector bundles}

Let $E$ be a vector bundle on a scheme $X$ and let $F \subset E$ be a subbundle For any effective Cartier divisor $D \subset X$, we define the \defi{elementary modification of $E$ at $D$ towards $F$} to be the kernel of the natural evaluation map
\[ E[D \to F] \colonequals \ker \left( E \to (E/F)|_D \right). \]
By \cite[Proposition~2.6]{aly}, $E[D \to F]$ is again a vector bundle, which is a subsheaf of \(E\).  

\begin{rem}\label{rem:factoring}
From this definition we see that an inclusion \(S \hookrightarrow E\) factors through \(E[D \to F]\) if and only if the restriction to \(D\) factors through \(F|_D\):
\[S|_D \hookrightarrow F|_D \hookrightarrow E|_D.\]
\end{rem}

\begin{rem}  Elementary modifications have a nice geometric interpretation in terms of projective bundles.  Suppose that \(E\) is a vector bundle of rank \(2\) on a smooth curve \(C\) and \(F\) is a line subbundle of \(E\).  In this case \(\pp F\) is a section of the \(\pp^1\)-bundle \(\pp E\) over \(C\).  The surface \(\pp E[p \to F]\) is obtained from \(\pp E\) by blowing up the point \(\pp F_p\) and blowing down the proper transform of the fiber \(\pp E_p\).  For more details see \cite[\S{}III.24]{beauville}.
\end{rem}

In the special case that \(F\) is a direct summand of \(E\), write
\(E \simeq F \oplus E'\).  Then
we see that \(E/F \simeq E'\), and so we have
\begin{equation}\label{eq:direct_sum}
E[D \to F] =  \ker \left( F \oplus E' \to E'|_D \right)  \simeq F \oplus E'(-D).\end{equation}
More generally, we can describe how elementary modifications play with respect to short exact sequences.  For simplicity we focus here on the rank \(2\) case that is of interest in this paper.  Suppose that \(C\) is a curve, \(S\) and \(Q\) are line bundles on \(C\) and \(E\) is a rank \(2\) bundle on \(C\) that sits in the exact sequence
\[0 \to S \to E \to Q \to 0.\]
Let \(p\) be a smooth point on \(C\).
Consider a line subbundle \(F\) of \(E\) and write \(k'\) for the order to which the fibers of \(S\) and \(F\) agree over \(p\) (i.e., the length of the support of \(\pp F \cap \pp S\) in \(\pp E\) in a neighborhood of \(p\)).  Let \(k\) be the minimum of \(k'\) and \(n\).
 Then we claim that the modification \(E[np \to F]\) sits in the exact sequence
 \begin{equation}\label{modifications_ses}
 0 \to S((k-n)p) \to E[np \to F] \to Q(-kp) \to 0.
 \end{equation}
This follows from  combining the observation that \(S((k-n)p) \hookrightarrow E[np \to F]\) is saturated with a Chern class computation to determine the twist at \(p\) in the quotient.
For a more detailed exposition on elementary modifications, we refer the reader to \cite[\S2--3]{aly}.

Let $q \in \pp^r$ be a point.
In this paper we will be primarily concerned with modifications of the normal bundle $N_{C/\pp^r}$ towards pointing bundles $N_{C \to q}$, which we now recall.  For a more detailed exposition see \cite[\S5--6]{aly}.  Write 
\[U_{C, q} = \{p \in C : T_pC \cap q = \emptyset\}.\]
Let $\pi_q \colon C \to \pp^{r-1}$ denote the projection map from $q$.
Note that $\pi_q|_{U_{C, q}}$ is unramified by construction.
If $U_{C, q}$ is dense in $C$ and contains the singular locus of $C$, then we may define $N_{C \to q}$ to be the unique extension to all of $C$ of the bundle
\[N_{C \to q}|_{U_{C, q}} \colonequals \ker \left( N_C|_{U_{C, q}} \to N_{\pi_q}|_{U_{C, q}} \right),\]
where $N_{\pi_q}$ denotes the normal sheaf of $\pi_q$. 
Our notation $N_{C \to q}$ is intended to suggest the geometry of sections: they point towards $q$ in $\pp^r$.
Projection from \(q\) induces an exact sequence
\begin{equation}\label{eq:proj_exact_seq}
0 \to N_{C \to q} \to N_C \to \pi_q^*N_{\pi_q}(C \cap q) \to 0.
\end{equation}
In this paper we will be primarily interested in the two simplest cases
\begin{enumerate}[(i)]
\item The point \(q \in \pp^r\) is general so that \(U_{C,q} = C\), and \(N_{C \to q} \simeq \O_C(1)\) by \cite[Proposition~6.2]{aly}, 
\item The point \(q\in C\) is general so that \(U_{C,q} = C \smallsetminus \{q\}\), and \(N_{C \to q} \simeq \O_C(1)(2q)\) by \cite[Proposition~6.3]{aly}.
\end{enumerate}
By convention, when modifying towards a pointing bundle, we will write
\[N_C[p \to q] \colonequals N_C[p \to N_{C \to q}].\]
The following foundational result of Hartshorne-Hirschowitz underpins our degenerative approach.

\begin{lem}[{\cite[Corollary~3.2]{HH}}]\label{lem:hh}
Let $X \cup Y$ be a connected nodal curve in $\pp^r$. Write $\{p_1, \dots, p_n\} = X \cap Y$ and let $q_i \in T_{p_i}Y$ be a choice of point.  Then
\[N_{X \cup Y}|_X \simeq N_X(p_1 + \cdots + p_n)[p_1 \to q_1] \cdots [p_n \to q_n].\]
\end{lem}

Given a vector bundle \(E\) on a reducible nodal curve \(X \cup Y\), restriction to the component \(X\) yields an exact sequence
\begin{equation}\label{eq:vb_rest}
 0 \to E|_Y(-X \cap Y) \to E \to E|_X \to 0.\end{equation}
If the vector bundle \(E\) is the normal bundle of the union \(N_{X \cup Y}\), then we can make this explicit using
Lemma \ref{lem:hh}.  Write \(q_{i,Y}\) for a choice of point in \(T_{p_i}Y\) and \(q_{i,X}\) for a choice of point in \(T_{p_i}X\).  Then \eqref{eq:vb_rest} yields the sequence
\begin{equation}\label{eq:rest_comp}
0 \to N_Y[p_1 \to q_{1,X}]\cdots [p_n \to q_{n,X}] \to N_{X \cup Y} \to N_X(p_1 + \cdots + p_n)[p_1 \to q_{1,Y}] \cdots [p_n \to q_{n,Y}] \to 0.\end{equation}

We will now illustrate how information about the normal bundle (such as presentations in exact sequences) can be combined with the data of modifications towards pointing bundles with the following three examples.  In future similar situations, we will point the reader to these examples and omit the details.

\begin{example}\label{ex:pointing_sub}
For a curve \(C \subset \pp^3\) and general points \(p, q \in C\),
consider the modified normal bundle \(N_C[p \to q]\).
Combining the sequence \eqref{eq:proj_exact_seq} coming from projection from \(q\) with \eqref{modifications_ses}, we see that this bundle sits in the exact sequence
\[0 \to \left[N_{C \to q} \simeq \O_C(1)(2q)\right] \to N_C[p \to q] \to \pi_q^* N_{\pi_q}(q-p) \to 0.\]
\end{example}

\begin{example}\label{ex:line}
Consider a line \(L \subset \pp^3\), points \(p_1, p_2 \in L\) and \(p_1', p_2' \in \pp^3\), such that the four points \(p_1, p_1', p_2, p_2'\) span \(\pp^3\).  Write \(H_i\) for the span of \(L\) and the point \(p_i'\).  The line \(L\) is the complete intersection of \(H_1\) and \(H_2\), and so
\[N_L \simeq N_{L/H_1} \oplus N_{L/H_2} \simeq \O_{\pp^1}(1) \oplus \O_{\pp^1}(1).\]
Furthermore, the pointing bundle \(N_{L \to p_i'}\) is isomorphic to \(N_{L/H_i}\) from which we conclude, using \eqref{eq:direct_sum}, that
\begin{align*}
N_L(p_1)[p_1 \to p_1'] & \simeq N_{L/H_1}(p_1) \oplus N_{L/H_2} \simeq \O_{\pp^1}(2) \oplus \O_{\pp^1}(1),\\
N_L(p_1 + p_2)[p_1 \to p_1'][p_2 \to p_2'] &\simeq N_{L/H_1}(p_1) \oplus N_{L/H_2}(p_2) \simeq \O_{\pp^1}(2) \oplus \O_{\pp^1}(2).
\end{align*}
\end{example}

\begin{example}\label{ex:enc}
Suppose that \(D\) is an elliptic normal curve (i.e., the complete intersection of two quadrics in \(\pp^3\)) and let \(p\) and \(q\) be general points on \(D\).  We will consider the modified normal bundle \(N_D[p \to q]\).  Since it is one condition for a quadric containing \(p\) and \(q\) to contain the line \(\bar{p,q}\), there must be a quadric \(Q\) in the pencil containing \(D\) that also contains the line \(\bar{p,q}\).  The line \(\bar{p,q}\) is contained in \(T_pQ\) and therefore the normal directions \(N_{D/Q}|_p\) and \(N_{D \to q}|_p\) coincide in \(N_D|_p\).  Combining the normal bundle exact sequence \eqref{eq:normal_bundle_exact_seq} for \(D \subset Q\) with \eqref{modifications_ses} yields the exact sequence of vector bundles
\[0 \to N_{D/Q} \to N_D[p \to q] \to N_Q|_D(-p) \to 0.\]
In fact, this exact sequence is split; choosing another quadric \(Q'\) from the pencil defining \(D\) gives the complement \(N_{D/Q'}(-p)\) to \(N_{D/Q}\).
Let \(L = \bar{p,q}\) be the \(2\)-secant line to \(D\) joining \(p\) and \(q\).  Combining the above discussion with Lemma \ref{lem:hh}, we see that
\[ N_{D \cup L}|_D \simeq N_D(p+q)[p \to q][q \to p] \simeq N_{D/Q}(p+q) \oplus N_{D/Q'}.\]
\end{example}

More generally, to deal with other curves on quadric surfaces and multiple modifications \([np \to q]\) 
in the course of our degenerations, we will make use of the following lemma, which computes that \(k=1\) in applying \eqref{modifications_ses}
to the normal bundle exact sequence.

\begin{lem}\label{lem:3pts}
Let $D$ be a (smooth) curve of type $(a,b)$ on a smooth quadric surface $Q$.
If $q$ is a general point of $D$, then
inside $\pp N_D$, the two sections coming from the line subbundles $N_{D \to q}$ and $N_{D / Q}$ meet transversely at $a+b-2$ points. 
\end{lem}
\begin{proof}
The fibers of $N_{D \to q}$ and $N_{D/Q}$ agree at $p$ if and only if $q$ is contained in $T_pQ$.  This occurs exactly at the points $p$ where the two lines through $q$ in $Q$ meet $D$.  Since $D$ is of type $(a,b)$ on $Q$, for $q$ general this happens at $a + b - 2$ points of $D$.

On the other hand, with multiplicity, the intersection number of these two sections is
\[c_1(N_D) - c_1(N_{D \to q}) - c_1(N_{D/Q}) = (2ab + 2a + 2b) - (a+ b +2) - (2ab) = a+b -2.\]
Therefore, when $q$ is general, these sections intersect transversely at exactly $a+ b -2$ points.
\end{proof}

It is a classical fact that the normal bundle of a rational normal (i.e.\ $(d, g) = (3, 0)$) or elliptic normal
(i.e.\ $(d, g) = (4, 1)$) curve is semistable,
which we record in the following lemma:

\begin{lem} \label{lem:01}
Let $C$ be a general Brill--Noether curve of degree $d$ and genus $g$,
where $(d, g) = (3, 0)$ or $(4, 1)$.
Then $N_C$ is semistable.
\end{lem}
\begin{proof} For $(d, g) = (3, 0)$,
let $p$ be a point on $C$, and write $\bar{C} \subset \pp^2$
for the image of $C$ under projection from $p$
(which is a conic). Then the semistability of $N_C$ follows from the
exact sequence
\[0 \to [N_{C \to p} \simeq \O_{\pp^1}(5)] \to N_C \to [N_{\bar{C}}(p) \simeq \O_{\pp^1}(5)] \to 0.\]

For $(d, g) = (4, 1)$, we note that $C$ is the complete intersection of two quadrics;
hence $N_C \simeq \O_C(2) \oplus \O_C(2)$ is semistable.
\end{proof}

\subsection{Modifications in families}\label{sec:mod_in_fam}
While arguing by degeneration, we will need the following technical
result, explained in Remark~3.4 of~\cite{aly}.
Suppose that we have a modification of a rank \(2\) vector bundle \(E\) towards two sub line bundles \(F_1\) and \(F_2\):
\[E[p_1 \to F_1][p_2 \to F_2].\]
Let \(p_1\) and \(p_2\) limit to a common point \(p\)
in a degeneration parameterized by a base \(B\).
More precisely, let \(p_1\) and \(p_2\) be sections
of \(C \times B \to B\) that intersect at \(p\) in the central fiber.
Write \(\pi \colon C \times B \to C\) for the projection, and
let \(F_1\) and \(F_2\) be subbundles of \(\pi^* E\).
\begin{itemize}
\item If \(F_1 = F_2 = F\), then the limit is \(E[2p \to F]\).
\item If \(F_1|_p\) and \(F_2|_p\) are linearly independent, then the limit
is \(E(-p)\).
\end{itemize}
This can be seen by constructing the modification
\((\pi^*E)[p_1 \to F_1][p_2 \to F_2]\) as a vector bundle on \(C \times B\),
and using that such modifications respect pullback \cite[Proposition 2.8]{aly}.
Sections~2 and~3 of \cite{aly} discuss this setup in much greater generality
for modifications that are treelike,
a condition that generalizes the assumption
that \(F_1 = F_2 = F\) or \(F_1|_p\) and \(F_2|_p\) are linearly independent.

We illustrate this here with an example.
Let \(C \subset \pp^3\)
be a curve, and \(p, q, u, v \in C\) be general points.
We consider the modified
normal bundle
\[N_C[p \to q][q \to p][u \to v][v \to u].\]
As we limit \(v\) to \(q\), the flat limit of these bundles is
\[N_C[p \to q][u \to q](-q).\]
If we further limit \(u\) to \(p\), then the flat limit is
\[N_C[2p \to q](-q).\]
(Note that this is not symmetric in \(p\) and \(q\); it depends on the
order of the limits.)

\subsection{Deformation theory of reducible curves}
In this section, we collect some basic facts about deformations
of reducible curves that we will need. For additional details,
the reader may consult a textbook on deformation theory
such as \cite{Hartshorne, Sernesi}.

Let \(C \subset \pp^r\) be any local complete intersection curve.
Write \(N_C = N_{C / \pp^r} = N_{C \hookrightarrow \pp^r}\) for the normal
bundle of \(C\), or equivalently for the normal sheaf of the inclusion
\(C \hookrightarrow \pp^r\). Then first-order deformations of \(C\)
are parameterized by \(H^0(N_C)\), and obstructions to lifting deformations
lie in \(H^1(N_C)\).

Now suppose that \(C\) is nodal, and write \(p\) for a node of \(C\).
We consider deformations of \(C\) that remain \defi{equisingular} at \(p\).
Equivalently, write \(\pi \colon \tilde{C} \to C\) for the partial normalization of \(C\) at \(p\),
and \(p_1\) and \(p_2\) for the points lying over \(p\) in \(\tilde{C}\).
Then equisingular deformations of \(C\) are deformations of the pointed
map \((\tilde{C}, p_1, p_2) \to \pp^r\) such that the deformations of
\(p_1\) and \(p_2\) map to the same point.
Such equisingular deformations are controlled by
a certain sheaf \(N\) on \(C\), that can be constructed in two ways:
\begin{enumerate}
\item 
The sheaf \(N\) can be constructed as the kernel of the natural map from \(N_C\)
to the deformation space \(T^1_{p^\infty}\) of the formal neighborhood
\(p^\infty\) of \(p\) in \(C\),
i.e.\ in symbols:
\[N = \ker(N_C \to T^1_{p^\infty}).\]
\item 
Alternatively, we can push forward the normal sheaf \(N_{\tilde{C} \to \pp^r}\)
of the map \(\tilde{C} \to \pp^r\) along the map \(\pi\). Evaluation at \(p_1\) and \(p_2\)
then defines a map
\[\pi_* N_{\tilde{C} \to \pp^r} \to \frac{T_p \pp^r}{T_{p_1} \tilde{C}} \oplus \frac{T_p \pp^r}{T_{p_2} \tilde{C}} \to \left(\frac{T_p \pp^r}{T_p C}\right)^2;\]
the sheaf \(N\) is then the preimage of the diagonal in \(\pi_* N_{\tilde{C} \to \pp^r}\).
\end{enumerate}
Then first-order deformations of \(C\) that fail to smooth the node \(p\) correspond
to \(H^0(N)\), and obstructions to lifting such deformations lie in \(H^1(N)\).

\begin{rem}
Away from \(p\), there is a natural isomorphism
between \(N\) and \(N_C\). Working locally,
a similar construction can be done for any subset of the nodes of \(C\).
Using the set of all nodes (so the deformations are equisingular
at all nodes rather than simply equisingular at \(p\)),
this construction appears in Section~4.7.1 of \cite{Sernesi}, in which
it is referred to as the \defi{equisingular normal sheaf}.
\end{rem}

Now suppose that \(C = X \cup Y\) is a reducible nodal curve,
with two smooth components \(X\) and \(Y\), and that \(p \in X \cap Y\)
is a node. Then the restrictions of \(N_C = N_{X \cup Y}\)
to \(X\) and \(Y\) are given in Lemma~\ref{lem:hh}.
Moreover, since \(N_{X \cup Y}\)
is a vector bundle on \(X \cup Y\), restriction to a component
defines an exact sequence (c.f.\ \eqref{eq:rest_comp}):
\[0 \to N_{X \cup Y}|_Y (-X \cap Y) \to N_{X \cup Y} \to N_{X \cup Y}|_X \to 0.\]
From either description given above for the sheaf \(N\),
we deduce analogous statements for \(N\). Namely,
\(N|_X\) is the modification of \(N_X\) where all modifications
appearing in Lemma~\ref{lem:hh} are performed
\emph{except} the one at \(p\).
The sequence for restriction to \(X\) takes a slightly different form
(since \(N\) is not a vector bundle in a neighborhood of \(p\)),
where the subbundle appearing in the sequence involves the restriction of
the ordinary normal bundle to \(Y\) (rather than the restriction of \(N\) to \(Y\)):
\begin{equation} \label{N-res}
0 \to N_{X \cup Y}|_Y (-X \cap Y) \to N \to N|_X \to 0.
\end{equation}

\subsection{Reducible Brill--Noether curves}\label{sec:defs}
In this section we show that the basic degenerations we will employ in the proof of Theorem \ref{thm:main} are in the Brill--Noether component of the Hilbert scheme.

We say that two curves $X$ and $Y$ meet \defi{quasi-transversely} at a set of points $\Gamma \subset \pp^r$ if for each $p \in \Gamma$, the tangent lines $T_pX$ and $T_p Y$ meet only in the isolated point $p$.  (If $r \geq 3$, two curves never meet transversely!)
The following Lemma is a special case of results of \cite{rbn},
but we include a characteristic-independent proof of this special case.

\begin{lem}\label{lem:rightcomp} Let $C$ be a general Brill--Noether curve of degree $d$ and genus $g$ and let $R$ be one of the following 
\begin{enumerate}[(i)]
\item\label{lem:rightcomp_1sec} a $1$-secant line meeting $C$ quasi-transversely at $p$,
\item\label{lem:rightcomp_2sec} a $2$-secant line meeting $C$ quasi-transversely at $p$ and $q$,
\item\label{lem:rightcomp_4sec} a $4$-secant conic meeting $C$ quasi-transversely at four coplanar points $p_1, \dots, p_4$.
\end{enumerate}
In case \eqref{lem:rightcomp_4sec} suppose that \(\rho(g,r,d) \geq 1\).
Then $C \cup R$ is a Brill--Noether curve of degree and genus \eqref{lem:rightcomp_1sec} $(d+1,g)$, \eqref{lem:rightcomp_2sec} $(d+1,g+1)$, \eqref{lem:rightcomp_4sec} $(d+2,g+3)$.
\end{lem}

\begin{proof}
By deformation theory, it suffices to show that $H^1(T_{\pp^3}|_{C \cup R})=0$, so that the map $C \cup R \to \pp^3$ may be lifted as $C \cup R$ is deformed to a general curve.  Moreover, if $C$ is general, then $H^1(T_{\pp^3}|_C) = 0$ by the Gieseker-Petri Theorem.
Using \eqref{eq:vb_rest}, we have an exact sequence 
\begin{equation}\label{rest_tang_es}0 \to T_{\pp^3}|_R(-R \cap C) \to T_{\pp^3}|_{C \cup R} \to T_{\pp^3}|_C \to 0.\end{equation}
In cases \eqref{lem:rightcomp_1sec} and \eqref{lem:rightcomp_2sec}, $T_{\pp^3}|_R \simeq \O(2) \oplus \O(1)^{\oplus 2}$.  Hence $H^1(T_{\pp^3}|_R(-p)) = 0$, respectively $H^1(T_{\pp^3}|_R(-p-q)) = 0$, and therefore, by \eqref{rest_tang_es} and the Gieseker-Petri Theorem for $C$, we have that $H^1(T_{\pp^3}|_{C\cup R}) = 0$.

For part \eqref{lem:rightcomp_4sec}, using that \(\rho(g,r,d) \geq 1\), we may apply part \eqref{lem:rightcomp_2sec} to specialize $C$ to the union of a Brill--Noether curve $C'$ of degree $d-1$ and genus $g-1$ and a $2$-secant line $L$, such that $R$ meets $C'$ at three points and meets $L$ at one point $p$.  Let $\Gamma \colonequals (L \cup R) \cap C'$, denoted by solid dots below.

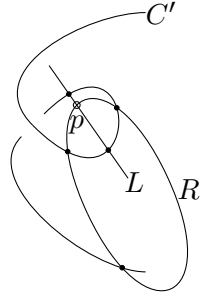
\vspace{-0.15in}

\begin{center}
\begin{tikzpicture}[scale=1.5]
\draw (1, 2) .. controls (0.5, 2) and (-0.5, 1.5) .. (0, 1);
\draw (0, 1) .. controls (1, 0) and (1, 2) .. (0.1, 1.1);
\draw (-0.1, 0.9) .. controls (-0.5, 0.5) and (0.5, -0.3) .. (1, -0.3);
\draw (0.145, 1.53) -- (0.845, 0.53);
\draw (0.4838, 0.22) .. controls (-0.0362, 1.3328) and (0.67, 1.6628) .. (1.19, 0.55);
\draw (0.4838, 0.22) .. controls (1.0038, -0.8928) and (1.71, -0.5628) .. (1.19, 0.55);
\draw (0.39, 1.18) circle[radius=0.03];
\draw (0.9, 0.5) node{$L$};
\draw (1.38, 0.43) node{$R$};
\draw (0.39, 1.01) node{$p$};
\draw (1.14, 2) node{$C'$};
\filldraw (0.793, -0.263) circle[radius=0.02];
\filldraw (0.67, 0.78) circle[radius=0.02];
\filldraw (0.31, 0.767) circle[radius=0.02];
\filldraw (0.32, 1.277) circle[radius=0.02];
\filldraw (0.745, 1.155) circle[radius=0.02];
\end{tikzpicture}
\end{center}

\vspace{-0.3in}

First, we show that (a) $C' \cup L \cup R$ is a smooth point of the Hilbert scheme and (b) we can smooth $L \cup R$ to a twisted cubic $R'$ that continues to pass through the $5$-points $\Gamma$. Let $N$ be the subsheaf of $N_{L \cup R}(-\Gamma)$ whose sections fail to smooth the node at $p$ (as discussed in Section \ref{sec:defs}).
Applying \eqref{N-res} in the case of restriction to $L$ gives the exact sequence 
\begin{equation}\label{eq:rest_to_L}
0 \to [N_{L \cup R}|_R(-p - \Gamma) \simeq \O\oplus \O(-1)] \to N \to [N|_L \simeq \O(-1)^{\oplus 2}] \to 0,
\end{equation}
where the isomorphisms within come from the explicit descriptions of \(R\) and \(L\) as complete intersections (as in Example \ref{ex:line}).
Hence, by the long exact sequence associated to \eqref{eq:rest_to_L}, we have $H^1(N) = 0$.
By deformation theory, statement (b) follows directly from $H^1(N) = 0$; this vanishing also implies $H^1(N_{C' \cup L \cup R})=0$ (and hence, by deformation theory, statement (a)).

To complete the proof, $T_{\pp^3}|_{R'}(-R'\cap C') \simeq \O(-1)^{\oplus 3}$ has no higher cohomology and so \eqref{rest_tang_es} and the Gieseker-Petri Theorem for $C'$ show that $H^1(T_{\pp^3}|_{C'\cup R'})=0$.  Therefore $C' \cup R'$ is in the Brill--Noether component.  Since $C' \cup L \cup R$ is a smooth point of the Hilbert scheme and both $C' \cup R'$ and $C \cup R$ are deformations of this, they are in the same component; in particular, $C \cup R$ is in the Brill--Noether component.
\end{proof}

\section{The Unstable Cases}\label{sec-unstable}

\subsection*{Arbitrary characteristic}
In two cases --- $(d, g) \in \{(5, 2), (6, 4)\}$ --- Theorem~\ref{thm:main}
asserts that, over a field of any characteristic, $N_C$ is unstable. In both of these cases, $C$ lies on a
quadric $Q$, and from the normal bundle exact sequence (c.f., \eqref{eq:normal_bundle_exact_seq}),
\begin{equation} \label{inQ}
0 \to [N_{C/Q} \simeq K_C(2)] \to N_C \to [N_Q|_C \simeq \O_C(2)] \to 0,
\end{equation}
we have that $N_C$ has a subbundle $N_{C/Q}$ of slope $2d + 2g - 2$.  If $(d,g) = (5,2)$ (respectively $(6,4)$) then $\mu(N_{C/Q}) = 12$ (respectively $18$), which is strictly more than $\mu(N_C) = 11$ (respectively $15$).

In fact, we can say more.
Note that $\operatorname{Ext}^1(\O_C(2), K_C(2)) \simeq H^1(K_C)$ is $1$-dimensional;
therefore there are only two such extensions up to isomorphism
(the split extension, and a unique nontrivial extension).

When $(d, g) = (6, 4)$, such curves $C$ are the complete intersection
of a quadric and cubic surface, and so \eqref{inQ} is split.
When $(d, g) = (5, 2)$, the following lemma
is equivalent to the assertion that \eqref{inQ} is nonsplit:

\begin{lem}\label{lem:52_sections_from_quadric}
Let $D$ be a Brill--Noether curve of degree $5$ and genus $2$ and let $Q$ be the unique quadric containing it.
The inclusion $K_D \simeq N_{D/Q}(-2) \subseteq N_D(-2)$ induces an isomorphism on global sections
\[H^0(K_D) \simeq H^0(N_D(-2)). \]
\end{lem}
\begin{proof}
As $H^0(K_D) \hookrightarrow H^0(N_D(-2))$, it suffices to show that $h^0(N_D(-2)) = 2$.
We will prove this by degenerating the curve $D$ to the union of an elliptic normal curve $E$ of degree $4$ and genus $1$ and a general $2$-secant line $L$ meeting $E$ quasi-transversely at $p$ and $q$, which is a Brill--Noether curve by Lemma \ref{lem:rightcomp}\eqref{lem:rightcomp_2sec}.

Since the tangent lines to \(E\) at \(p\) and \(q\) span \(\pp^3\), combining Lemma \ref{lem:hh} with Example \ref{ex:line}, we see that
$N_{E \cup L}(-2)|_L \simeq \O_L \oplus \O_L$
has a $2$-dimensional space of  global sections.   Furthermore, since $H^0(N_{E \cup L}(-2)|_L(-p-q)) = 0$, we have that
\begin{equation}\label{union_inclusion}
H^0(N_{E\cup L}(-2)) \hookrightarrow H^0(N_{E \cup L}(-2)|_{E}). 
\end{equation}

As in Example \ref{ex:enc}, choosing quadrics \(Q_1\) and \(Q_2\) whose intersection is \(E\) and such that \(Q_1\) contains \(L\), we see that
the normal bundle restricted to $E$
\[N_{E\cup L}(-2)|_E \simeq N_{E/Q_1}(-2)(p + q) \oplus N_{E/Q_2}(-2) \simeq \O_E(p+q) \oplus \O_E, \]
has a $3$-dimensional space of global sections.  It remains to show that one of these sections is not in the image of \eqref{union_inclusion}.

We claim that the unique (up to scaling) section of $\O_E$ is not in the image of \eqref{union_inclusion}.
Indeed, since $L$ is transverse to $Q_2$, this section fails to smooth both nodes;
if it extended across $L$, it must extend to a section in $H^0(N_L(-2)) \subset H^0(N_{E \cup L}|_L(-2))$.
But $N_L(-2) \simeq \O_L(-1) \oplus \O_L(-1)$ has no global sections, so any extension across would have to vanish
identically along $L$, and in particular at $p$ and $q$
(which this section does not).
\end{proof}

\subsection*{Characteristic 2}
Theorem \ref{thm:main} asserts that, in characteristic $2$, there are infinitely many pairs $(d,g) = (2k,0)$ for which the normal bundle of a general Brill--Noether curve is unstable.  This is the first case of a more general phenomena occurring only in characteristic $2$.

Let $C \subset \pp^r$ be a Brill--Noether curve.  In any characteristic, the Euler sequence \eqref{euler_seq} shows that the bundle $N_C^\vee(1)$ sits in an exact sequence
\begin{equation}\label{prin_parts} 0 \to N_C^\vee(1) \to \O_C^{\oplus r+1} \to \sP^1(\O_C(1)) \to 0,\end{equation}
where $\sP^1(\O_C(1))$ is the first bundle of principal parts of the line bundle $\O_C(1)$.

Now assume that $\chara(k) =2$ and let $\pi \colon C \to C^{(2)}$ denote the (relative) Frobenius morphism.  Given a reduced point $c \in C$, the fiber of $\pi$ containing $c$ is the nonreduced point $2c$.  Therefore 
\[\sP^1(\O_C(1)) \simeq \pi^*\pi_* \O_C(1).\]
Thus $N_C^\vee(1) \simeq \pi^* K$ is isomorphic to the pullback of a vector bundle $K$ under Frobenius.  Using this, we have the following.

\begin{lem}\label{lem:char2_obs}
Assume that $\chara(k)=2$ and let $C \simeq \pp^1$ be a rational curve of degree $d$ in $\pp^r$ over $k$.  Then the normal bundle splits as
\[N_C \simeq \bigoplus_i \O_{\pp^1}(a_i), \]
for integers $a_i \equiv d \pmod{2}$.
\end{lem}
\begin{proof}
If $\chara(k)=2$, then $N_C^\vee(1) \simeq \pi^*K$ for some vector bundle $K$ on $\pp^1$. Write $K \simeq \bigoplus \O_{\pp^1}(k_i)$.
Since $\pi^* \O_{\pp^1}(a) \simeq \O_{\pp^1}(2a)$, we have
$N_C \simeq \bigoplus \O_{\pp^1}(d - 2k_i)$ as desired.
\end{proof}

\begin{cor}\label{cor:char2_stab}
Let $C$ be a general rational curve in $\pp^r$ of degree $d \geq r$.
Then $N_C$ is semistable only if $2d \equiv 2 \pmod{r-1}$;
in characteristic~$2$, this can be strengthened to $d \equiv 1 \pmod{r - 1}$.
\end{cor}
\begin{proof}
In any characteristic, $N_C$ can only be semistable if $\mu(N_C) = d + \frac{2d-2}{r-1}$ is an integer.
In characteristic~$2$,
Lemma~\ref{lem:char2_obs} implies that furthermore $\mu(N_C) - d$ must be an \emph{even} integer.
\end{proof}

\begin{rem}
When $r=3$, we prove in Section
\ref{sec:reduction_finite}
that Corollary \ref{cor:char2_stab} gives the only obstruction to semistability for the normal bundle of a rational curve in characteristic $2$.
With a little more work, one can show the same in any projective space.
\end{rem}

\section{Stability and degeneration I}\label{sec:stab_degen1}

In this section, by specializing to the union of a general Brill--Noether curve and a $4$-secant conic, we reduce Theorem~\ref{thm:main} to the cases $g \leq 8$.  Our main tool will be the following first basic lemma proving stability by degeneration.


\begin{lem} \label{lem:naive}
Suppose that $C  = X \cup Y$ is a reducible curve and $E$ is a vector bundle on $C$ such that $E|_X$ and $E|_Y$ are semistable.
Then $E$ is semistable.
Furthermore, if one of $E|_X$ or $E|_Y$ is stable, then $E$ is stable.
\end{lem}
\begin{proof}
Write $\nu \colon \tilde{X} \sqcup \tilde{Y} \to C$ for the normalization map.
For any subbundle $F \subseteq \nu^* E$ we have
\[\mu^{\text{adj}} \left( F \right) \leq  \mu^{\text{adj}}_X\left( F|_{\tilde{X}} \right) + \mu^{\text{adj}}_Y\left( F|_{\tilde{Y}} \right) \leqpar \mu\left( E|_X \right) + \mu\left( E|_Y \right) =  \mu\left( E \right).  \qedhere\]
\end{proof}

\subsection*{\boldmath $4$-secant conic degenerations} 
Let $C$ be a Brill--Noether curve of degree $d \geq 4$ and genus $g$ in~$\pp^3$.  Let $H \subset \pp^3$ be a $2$-plane meeting $C$ transversely; let $p_1, \dots, p_4$ be four points in $C\cap H$.  For $R \subset H$ a conic through $p_1, \dots, p_4$, the union $C\cup R$ is a Brill--Noether curve of degree $d+2$ and genus $g + 3$ by Lemma \ref{lem:rightcomp}\eqref{lem:rightcomp_4sec}, provided that \(\rho(g,3, d) \geq 1\).

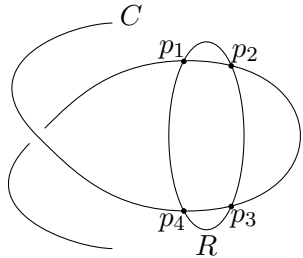
\begin{center}
\begin{tikzpicture}[scale=1]
\draw (1, 3.5) .. controls (0.5, 3.5) and (-1, 3) .. (0, 2);
\draw (0, 2) .. controls (0.5, 1.5) and (1, 1) .. (2, 1);
\draw (0.1, 2.1) .. controls (0.5, 2.5) and (1, 3) .. (2, 3);
\draw (1, 0.5) .. controls (0.5, 0.5) and (-1, 1) .. (-0.1, 1.9);
\draw (2, 1) .. controls (4, 1) and (4, 3) .. (2, 3);
\draw (2.25, 2) ellipse (0.5 and 1.25);
\filldraw (1.95,2.99) circle[radius=0.03];
\filldraw (1.95,1) circle[radius=0.03];
\filldraw (2.58,2.93) circle[radius=0.03];
\filldraw (2.58,1.06) circle[radius=0.03];
\draw (1.8,3.15) node{$p_1$};
\draw (1.8,0.85) node{$p_4$};
\draw (2.77,3.1) node{$p_2$};
\draw (2.75,0.9) node{$p_3$};
\draw (1.25,3.6) node{$C$};
\draw (2.25,.55) node{$R$};
\end{tikzpicture}
\end{center}

\begin{lem}\label{lem:4_sec_conic}
In the above setup, if $C$ is a general Brill--Noether curve with $(d,g) \neq (3, 0)$ or $(4,1)$, then
\[N_{C \cup R}|_R \simeq \O_{\pp^1}(5) \oplus \O_{\pp^1}(5) \]
is semistable.
\end{lem}
\begin{proof}
We will prove this lemma by degeneration of $C$.
If $C$ admits a degeneration to $X \cup Y$, where $\deg X \geq 4$, then we may consider degenerations $X \cup Y \cup R$ of $C \cup R$ where the conic $R$ meets $X$ alone; this reduces the case of $C$ to the case of $X$.

By repeatedly applying Lemma~\ref{lem:rightcomp} to
pull off $1$-secant lines, $2$-secant lines,
or $4$-secant conics, we thus reduce to the case where
$(d, g)$ satisfies
\begin{equation} \label{foo}
\rho(g,3,d) \geq 0, \quad g \geq 0, \quad \text{and} \quad (d, g) \neq (3, 0), (4, 1),
\end{equation}
but $(d', g')$ fails to satisfy \eqref{foo}
for each of $(d', g') = (d - 1, g), (d - 1, g - 1)$, and $(d - 2, g - 3)$.

By inspection, this is only possible if $(d, g) = (4, 0), (5, 2),$ or $(6, 4)$.
(Indeed, if $g \geq 5$, then $(d', g') = (d - 2, g - 3)$
satisfies \eqref{foo};
if $g \leq 4$ and $d \geq 7$, then $(d', g') = (d - 1, g)$
satisfies \eqref{foo};
the finitely many cases with $g \leq 4$ and $d \leq 6$ are easily
verified.)

In these cases, $C$ is of type $(3, d - 3)$ on a quadric.
Specializing $C$ to the union of a curve of type $(3, 1)$
with $d - 4$ lines of type $(0, 1)$,
it thus remains only to consider the case $(d, g) = (4, 0)$.

When $C$ is a rational quartic curve, we specialize $C$ to
$C' \cup L$ where $C'$ is a rational normal curve and $L$ is a $1$-secant
line meeting $C'$ at a point $x$. Since $C$ has degree $4$, we must specialize $R$ to meet $L$
in one point $y$ and $C'$
in a set $\{z_1, z_2, z_3\}$
of three points:

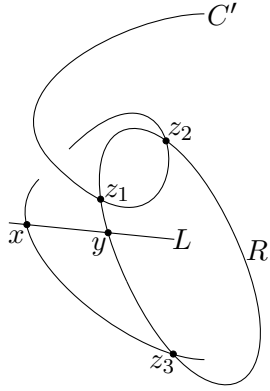
\begin{center}
\begin{tikzpicture}[scale=2]
\draw (1, 2) .. controls (0.5, 2) and (-0.5, 1.5) .. (0, 1);
\draw (0, 1) .. controls (1, 0) and (1, 2) .. (0.1, 1.1);
\draw (-0.1, 0.9) .. controls (-0.5, 0.5) and (0.5, -0.3) .. (1, -0.3);
\draw (0.4838, 0.22) .. controls (-0.0362, 1.3328) and (0.67, 1.6628) .. (1.19, 0.55);
\draw (0.4838, 0.22) .. controls (1.0038, -0.8928) and (1.71, -0.5628) .. (1.19, 0.55);
\draw (0.85, 0.5) node{$L$};
\draw (1.35, 0.43) node{$R$};
\draw (1.12, 2) node{$C'$};
\filldraw (0.793, -0.263) circle[radius=0.02];
\filldraw (0.31, 0.767) circle[radius=0.02];
\filldraw (0.745, 1.155) circle[radius=0.02];
\draw (0.8, 0.5) -- (-0.3, 0.61);
\filldraw (0.36, 0.544) circle[radius=0.02];
\filldraw (-0.18, 0.598) circle[radius=0.02];
\draw (0.72, -0.33) node{$z_3$};
\draw (0.3, 0.45) node{$y$};
\draw (-0.25, 0.52) node{$x$};
\draw (0.42, 0.81) node{$z_1$};
\draw (0.84, 1.22) node{$z_2$};
\end{tikzpicture}
\end{center}

Since $N_{C'} \simeq \O_{\pp^1}(5) \oplus \O_{\pp^1}(5)$,
we may arrange for $C'$ to have general tangent directions
at the points $z_i$. Thus, $N_{C' \cup R}|_R \simeq \O_{\pp^1}(5) \oplus \O_{\pp^1}(4)$.
In particular, we have a distinguished subspace of $N_R|_y$ given by the positive subbundle
$\O_{\pp^1}(5)|_y \subset N_{C' \cup R}|_y \simeq N_R|_y$ ---
or equivalently, a distinguished plane $\Lambda \supset T_y R$.
Since $x \in C'$ is general, we have $x \notin \Lambda$.
Thus
\[N_{C' \cup L \cup R}|_R \simeq N_{C' \cup R}|_R(y)[y \to x] \simeq \O_{\pp^1}(5) \oplus \O_{\pp^1}(5). \qedhere\]
\end{proof}

\begin{rem}
For $(d, g) = (4, 1)$, the conclusion
of Lemma~\ref{lem:4_sec_conic} is false: For any $R$, the curve $C$ lies on a quadric $Q$ containing $R$, and $N_{(C \cup R)/Q}|_R$ is destabilizing.
\end{rem}

\noindent
Let $p_i'$ be a point on $T_{p_i}R \smallsetminus p_i$.  Then by Lemma \ref{lem:4_sec_conic} combined with Lemma~\ref{lem:naive}, stability for 
\[N_C[p_1 \to p_1'][p_2 \to p_2'][p_3 \to p_3'][p_4 \to p_4']\]
implies stability for $N_{C\cup R}$, and hence for the normal bundle of a general Brill--Noether space curve of degree $d+2$ and genus $g+3$.

\subsection*{\boldmath Deformations of $r$-secant rational curves}

In our application of the above degeneration to reduce to a finite list of genera, we will specialize to the union of a Brill--Noether \(D\) and two quasitransverse \(4\)-secant conics through \emph{the same} set of \(4\) points.  
To employ this degeneration, we must know that such conics can be suitably deformed while preserving the incidence conditions with $D$.

In greater generality, let 
$D$ be a Brill--Noether curve, and $R$ be a rational curve meeting $D$
at distinct points $p_1, p_2, \ldots, p_r$.
%
The following key assumption generalizes the conclusion of
Lemma~\ref{lem:4_sec_conic}:

\begin{assumption}\label{assump:balanced_slope}
The restricted normal bundle
$N_{D\cup R}|_R$
is perfectly balanced with slope
\[\mu(N_{D \cup R}|_R) \geq r + 1.\]
\end{assumption}

\begin{lem}\label{lem:deform_delta}
Under assumption \ref{assump:balanced_slope}, there exists
a deformation $R(t)$ of $R$, and $p_i(t)$ of $p_i$,
such that the rational curve $R(t)$ meets $D$ quasi-transversely in $p_1(t), p_2(t), \ldots, p_r(t)$,
and $p_i(t)$ has nonzero derivative at $t = 0$ for all $i$.
\end{lem}
%

\begin{proof}
For any $i$, let $N_i$ denote the vector bundle on $R$ obtained by making elementary modifications
to \(N_R\) at all points of \(D \cap R\) except \(p_i\) in the direction of \(D\) (i.e.\ the vector
bundle obtained by gluing the vector bundles
$N_{R \cup D}|_{R \smallsetminus p_i}$ and $N_R|_{R \smallsetminus \{p_1, \ldots, \hat{p_i}, \ldots, p_r\}}$ along the natural isomorphism 
$N_{R \cup D}|_{R \smallsetminus \{p_1, \ldots, p_r\}} \simeq N_R|_{R \smallsetminus \{p_1, \ldots, p_r\}}$).  This bundle \(N_i\) controls the deformations of \(D \cup R\) along \(D\) that remain equisingular at \(p_i\) (c.f., the discussion in Section~\ref{sec:defs}).
Obstructions to lifting deformations of $p_i$ to deformations of $R$ that preserve the incidence conditions
with $D$ at the $p_j$
lie in $H^1(N_i(-p_1-\cdots-p_r))$; it thus suffices to show
\[H^1(N_i(-p_1-\cdots - p_r)) = 0.\]
The bundle $N_i(-p_1-\cdots - p_r)$ fits in an exact sequence
\[0 \to N_{R \cup D}|_R(-p_1- \cdots - p_{i - 1} - 2 p_i - p_{i+1} - \cdots - p_r) \to N_i(-p_1-\cdots - p_r) \to \O_{p_i} \to 0,\]
The long exact sequence of cohomology implies the desired vanishing since by assumption 
$N_{D \cup R}|_R$ is perfectly balanced with slope
$\mu(N_{D \cup R}|_R) \geq r + 1$, hence
\[H^1(N_{R \cup D}|_R(-p_1- \cdots - p_{i - 1} - 2 p_i - p_{i+1} - \cdots - p_r)) = 0. \qedhere\] 
\end{proof}

\subsection*{Reduction to a finite list of genera}

\begin{lem}\label{lem:pancake}
Suppose that Theorem \ref{thm:main} is true for all $g\leq 8$.  Then it is true for all $g$.
\end{lem}
\begin{proof}
If $\rho(g,3,d) \geq 0$ and $g \geq 9$, then
\begin{equation}\label{rho}\rho(g-6, 3, d-4) = \rho(g,3, d) + 2 \geq 0 \quad \text{and} \quad  \ g - 6 \geq 2 \quad \text{and} \quad (d - 4, g - 6) \notin \{(5, 2), (6, 4)\},\end{equation}
\begin{equation}\label{dminus4}
\text{and}\qquad d-4 \geq 4.
\end{equation}

By \eqref{rho}, a general Brill--Noether curve $D$ of degree $d-4$ and genus $g-6$ has $N_D$ stable by induction. Let $H$ be a general hyperplane; by \eqref{dminus4}, we may let $R_1 \subseteq H$ and $R_2 \subseteq H$ be general $4$-secant conics,
both of which meet $D$ at $p_1, \ldots, p_4$:

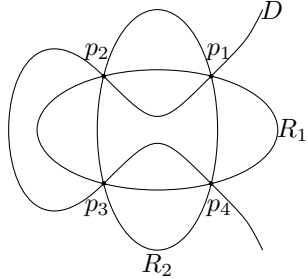
\begin{center}
\begin{tikzpicture}[scale=.8]
{\small
\draw (0, 0) ellipse (2 and 1);
\draw (0, 0) ellipse (1 and 2);
\filldraw (0.892, 0.892) circle[radius=0.03];
\filldraw (0.892, -0.892) circle[radius=0.03];
\filldraw (-0.892, 0.892) circle[radius=0.03];
\filldraw (-0.892, -0.892) circle[radius=0.03];
\draw (0.892, 0.892) .. controls (0, 0) .. (-0.892, 0.892);
\draw (0.892, -0.892) .. controls (0, 0) .. (-0.892, -0.892);
\draw (-0.892, 0.892) .. controls (-3, 3) and (-3, -3) .. (-0.892, -0.892);
\draw (0.892, 0.892) .. controls (1.5, 1.5) .. (1.75, 2);
\draw (0.892, -0.892) .. controls (1.5, -1.5) .. (1.75, -2);
\draw (0, -2.25) node{$R_2$};
\draw (2.25, 0) node{$R_1$};
\draw (1.9, 2) node{$D$};
\draw (1.03, 1.3) node{$p_1$};
\draw (1.03, -1.3) node{$p_4$};
\draw (-1, 1.3) node{$p_2$};
\draw (-1, -1.3) node{$p_3$};
}
\end{tikzpicture}
\end{center}

By Lemma~\ref{lem:deform_delta}, we may deform $R_i$ to $4$-secant
conics $R_i(t)$ meeting $D$ at $p_{i1}(t)$, $p_{i2}(t)$, $p_{i3}(t)$, and $p_{i4}(t)$,
such that $p_{1j}(t)$ and $p_{2j}(t)$ have distinct derivatives:
\begin{center}
\begin{tikzpicture}[scale=0.5]
\draw (0, 0) ellipse (2 and 1);
\draw[densely dotted] (0, 0) ellipse (2.2 and 1.1);
\draw (0, 0) ellipse (1 and 2);
\draw[densely dotted] (0, 0) ellipse (0.9 and 1.8);
\draw (0.892, 0.892) .. controls (0, 0) .. (-0.892, 0.892);
\draw (0.892, -0.892) .. controls (0, 0) .. (-0.892, -0.892);
\draw (-0.892, 0.892) .. controls (-3, 3) and (-3, -3) .. (-0.892, -0.892);
\draw (0.892, 0.892) .. controls (1.5, 1.5) .. (1.75, 2);
\draw (0.892, -0.892) .. controls (1.5, -1.5) .. (1.75, -2);
\end{tikzpicture}
\end{center}
Combining lemmas~\ref{lem:naive} and \ref{lem:4_sec_conic},
it remains to show the stability of
$N_C[p_{ij}(t) \to p_{ij}'(t)]$ for $t \in \Delta$ general,
where $p_{ij}'(t)$ denotes a point on $T_{p_{ij}(t)} C \smallsetminus p_{ij}(t)$.
By the discussion in Section \ref{sec:mod_in_fam}, these vector bundles fit together
to form a vector bundle over $D \times \Delta$ whose fiber over $0 \in \Delta$
is the bundle $N_D(-p_1-p_2-p_3-p_4)$ --- which is stable since we have already seen that
$N_D$ is stable by induction.
\end{proof}

%

\section{Stability and Degeneration II: Gluing Data}\label{sec-gluing}

In order to settle the base cases $g \leq 8$, we will need to use degenerations of $C$
to reducible curves $X \cup Y$ where neither $N_{X \cup Y}|_X$ nor $N_{X \cup Y}|_Y$
are necessarily stable.
The basic idea is to compare destabilizing subbundles of $N_{X \cup Y}|_X$ and $N_{X \cup Y}|_Y$,
and show that they cannot agree sufficiently over $X \cap Y$.

\subsection*{\boldmath $1$-secant degenerations}

In some cases, 
we can construct a \emph{modification} of the restriction $N_{X \cup Y}|_X$
whose stability rules out a destabilizing subbundle of $N_{X \cup Y}|_X$ that could agree sufficiently with a destabilizing
subbundle of $N_{X \cup Y}|_Y$. 
This technique works well when we can understand the geometry of $Y$ explicitly.
Here we apply this technique
when $Y = L$ is a $1$-secant line.

Let $D$ be a smooth Brill--Noether curve and $L$ a quasi-transverse $1$-secant line meeting $D$ at $p$.
Although $N_{D \cup L}|_L$ is not semistable, so we cannot apply Lemma~\ref{lem:naive},
we can identify the unique destabilizing subbundle of $N_{D \cup L}|_L$,
and construct a modification of $N_{D \cup L}|_D$ as described above.

For inductive arguments it will be more useful to consider a slightly more general setup:
Let $N_{D \cup L}'$ be any vector bundle equipped with an isomorphism
with $N_{D \cup L}$ over an open set $U$ of $D \cup L$ containing $L$,
and write $N_D'$ for the bundle obtained by gluing $N_D|_U$ to $N_{D \cup L}'|_{D \smallsetminus p}$
along the isomorphism
$N_D|_{U \smallsetminus p} \simeq N_{D \cup L}|_{U \smallsetminus p} \simeq N_{D \cup L}'|_{U \smallsetminus p}$.
To state the lemma, let $q \in L \smallsetminus p$.

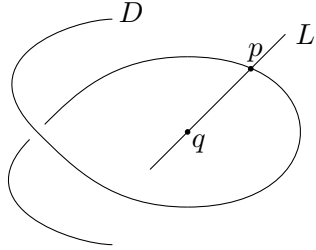
\begin{center}
\begin{tikzpicture}[scale=1]
\draw (1, 3.5) .. controls (0.5, 3.5) and (-1, 3) .. (0, 2);
\draw (0, 2) .. controls (0.5, 1.5) and (1, 1) .. (2, 1);
\draw (0.1, 2.1) .. controls (0.5, 2.5) and (1, 3) .. (2, 3);
\draw (1, 0.5) .. controls (0.5, 0.5) and (-1, 1) .. (-0.1, 1.9);
\draw (2, 1) .. controls (4, 1) and (4, 3) .. (2, 3);
\draw (1.25,3.6) node{$D$};
\draw (3.3,3.3) -- (1.5,1.5);
\draw (3.3,3.3) node[right]{$L$};
\filldraw (2.84,2.84) circle[radius=0.03];
\draw (2.9,3.05) node{$p$};
\filldraw (2,2) circle[radius=0.03];
\draw (2.15,1.85) node{$q$};
\end{tikzpicture}
\end{center}

\begin{lem}\label{lem:1_sec_balanced}
In the above setup, 
if $N_D'[p \to q][p \to q] \simeq N'_D[2p \to q]$ is (semi)stable,
then $N'_{D \cup L}$ is also (semi)stable.
\end{lem}
\begin{proof}
Write $\nu \colon D \sqcup L \to D \cup L$ for the normalization map,
and $\tilde{p}_1$ and $\tilde{p}_2$ for the points above $p$ on $L$ and $D$
respectively.
Suppose that $F \subseteq \nu^* N'_{D\cup L}$ is a line subbundle.

First, we consider the restriction of $F$ to $L$.
Let $x$ be a point on $T_pD$ and let $\Lambda$ be the plane spanned by $x$ and $L$.  Let $H$ be another plane such that $L = \Lambda \cap H$.  Then by Lemma \ref{lem:hh} and Example \ref{ex:line}, 
\[N_{D\cup L}'|_L \simeq N_L(p)[p \to x] \simeq  N_{L/H} \oplus N_{L/\Lambda}(p) \simeq  \O_{\pp^1}(1) \oplus \O_{\pp^1}(2).\]
Consequently,
\begin{equation} \label{fl}
\mu(F|_L) \leq \begin{cases}
2 & \text{if $F|_{\tilde{p}_1} = N_{L/\Lambda}(p)|_{\tilde{p}_1}$;} \\
1 & \text{otherwise.}
\end{cases}
\end{equation}

Second, we consider the restriction of $F$ to $D$.
If $F|_{\tilde{p}_2} = N_{D \to q}(p)|_{\tilde{p}_2}$, then, by Remark \ref{rem:factoring}, $F|_D$ is a subbundle
of $N'_{D \cup L}|_D[p \to q] \simeq N'_D(p)[2p \to q]$;
otherwise $F|_D(-\tilde{p}_2)$ is a subbundle of $N'_D(p)[2p \to q]$. Because $N'_D[2p \to q]$
is (semi)stable by assumption and of slope $\mu(N_D') - 1$, it follows that $N'_D(p)[2p \to q]$
is (semi)stable of slope $\mu(N_D')$. Consequently,
\begin{equation} \label{fd}
\mu(F|_D) \leqpar \begin{cases}
\mu(N'_D) + 1 & \text{if $F|_{\tilde{p}_2} \neq N_{D \to q}(p)|_{\tilde{p}_2}$;} \\
\mu(N'_D) & \text{otherwise.}
\end{cases}
\end{equation}

Finally, by \cite[Lemma 8.5]{aly}, the subspace $N_{L/\Lambda}(p)|_{\tilde{p}_1}$ glues to the subspace $N_{D \to q}(p)|_{\tilde{p}_2}$.
Consequently,
\begin{equation} \label{fp}
\codim_{F} (F|_{\tilde{p}_1} \cap F|_{\tilde{p}_2}) \geq \begin{cases}
1 & \text{if $F|_{\tilde{p}_1} = N_{L/\Lambda}(p)|_{\tilde{p}_1}$ and $F|_{\tilde{p}_2} \neq N_{D \to q}(p)|_{\tilde{p}_2}$;} \\
0 & \text{otherwise.}
\end{cases}
\end{equation}
To finish the proof,
we simply combine \eqref{fl}, \eqref{fd}, and \eqref{fp}, to obtain
\[\mu^{\text{adj}}(F) = \mu(F|_L) + \mu(F|_D) - \codim_{F} (F|_{\tilde{p}_1} \cap F|_{\tilde{p}_2}) \leqpar \mu(N'_D) + 2 = \mu(N'_{D \cup L}). \qedhere\]
\end{proof}


\begin{lem}\label{lem:1_sec}
Assume that the characteristic of the ground field is not $2$.  Suppose that $N_D$ is (semi)stable.  If $q \in \pp^3$ is a general point and $p \in D$ has ordinary ramification, then the elementary modification $N_D[2p \to q]$ is (semi)stable.
\end{lem}
\begin{proof}
Let $\Lambda\subset \pp^3$ be a $2$-plane containing $T_pD$ that is not the osculating $2$-plane to $D$ at $p$.  For parameter $s \in \pp^1$, let $L_s$ be the pencil of lines through $p$ in $\Lambda$ specializing to $T_pD$ when $s=0$ and let $q(s)$ be a choice of point on $L_s \smallsetminus p$.

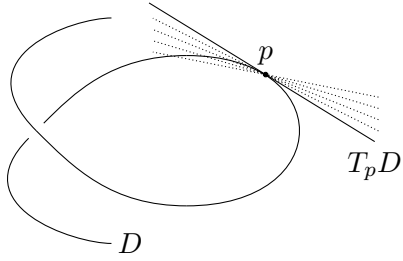
\begin{center}
\begin{tikzpicture}[scale=1]
\draw (1, 3.5) .. controls (0.5, 3.5) and (-1, 3) .. (0, 2);
\draw (0, 2) .. controls (0.5, 1.5) and (1, 1) .. (2, 1);
\draw (0.1, 2.1) .. controls (0.5, 2.5) and (1, 3) .. (2, 3);
\draw (1, 0.5) .. controls (0.5, 0.5) and (-1, 1) .. (-0.1, 1.9);
\draw (2, 1) .. controls (4, 1) and (4, 3) .. (2, 3);
\draw (1.25,.5) node{$D$};
\draw (1.5,3.7)  -- (4.5, 1.855);
\draw (4.5, 1.855) node[below]{$T_pD$};
\filldraw (3.05,2.74675) circle[radius=0.03];
\draw (3.05,2.74675) node[above]{$p$};
\draw[densely dotted] (3.05 +1.5, 1.99675) -- (3.05-1.5, 3.49675);
\draw[densely dotted] (3.05 +1.5, 2.14675) -- (3.05-1.5, 3.34675);
\draw[densely dotted] (3.05 +1.5, 2.29675) -- (3.05-1.5, 3.19675);
\draw[densely dotted] (3.05 +1.5, 2.44675) -- (3.05-1.5, 3.04675);
\end{tikzpicture}
\end{center}

As (semi)stability is open, and $N_D(-p)$ is (semi)stable by assumption,
it suffices to show that the modifications $N_D[2p \to q(s)]$ for $s\neq 0$ fit together into a flat family specializing to $N_D(-p)$ when $L_s = T_pD$.
To do this, we first observe that, for $s\neq 0$,
\[N_D[2p \to q(s)] \colonequals \ker \left( N_D \to \frac{N_D|_{2p}}{N_{D \to q(s)}|_{2p}} \right) \]
is determined by the $2$-dimensional subspace $N_{D \to q(s)}|_{2p}$  of the $4$-dimensional space $N_D|_{2p}$.  As the Grassmannian $\operatorname{Gr}(2,4)$ is separated and proper, there is a unique limit of these spaces as $s \to 0$.
It suffices to prove, by a calculation in local coordinates,
that this subspace is $N_D(-p)|_{2p} \subseteq N_D|_{2p}$.

Choose an affine neighborhood $ \aa_{xyz}^3 \subseteq \pp^3$ and coordinates such that $p = (0,0,0)$, the tangent line $T_pD$ is $y = z = 0$, the osculating two-plane is $z = 0$, and $\Lambda$ is $y = 0$.  Let $q(s) = (1, 0, s)$ so that $L_s$ is the line through $(1,0,s)$ and $(0,0,0)$.

Let $t$ be an \'etale local coordinate at $p$ for $D$.  Then in an \'etale neighborhood of $p$, the curve $D$ is given parametrically by
\[D(t) = \begin{pmatrix} t \\ t^2 + a_3 t^3 + \cdots \\ b_3 t^3 + \cdots \end{pmatrix}.\]
We trivialize $N_D$ in a neighborhood of $p$ by $\del/\del y$ and $\del/\del z$.  A section of $N_D$ is then given by
\begin{equation}\label{section_mn} (m_0 + m_1 t + m_2 t^2 + \cdots ) \frac{\del}{\del y} +  (n_0 + n_1 t + n_2 t^2 + \cdots ) \frac{\del}{\del z}. \end{equation}
We must determine the conditions on the $m_i$ and $n_i$ such that this section points towards $q(s)$ to second order in $t$.  The vector from $D(t)$ on $D$ to $q(s)$
\[ D(t) - q(s)= \begin{pmatrix} t-1 \\ t^2 + a_3t^3 + \cdots \\ b_3t^3 + \cdots - s \end{pmatrix}\]
is equivalent as a section of $N_D$ to its translate by a tangent vector
\[ D(t) - q(s) - (t-1)D'(t) = \begin{pmatrix} t-1 \\ t^2 + a_3t^3 + \cdots \\ b_3t^3 + \cdots - s \end{pmatrix} -  \begin{pmatrix} t-1 \\ (t-1)(2t + 3a_3t^2 + \cdots) \\ (t-1)(3b_3t^2 + \cdots) \end{pmatrix} = \begin{pmatrix} 0 \\ 2t + (3a_3 - 1)t^2 + \cdots \\ -s - 3b_3t^2 + \cdots \end{pmatrix}. \]
This normal vector now corresponds to the section
\[(2t + (3a_3 - 1)t^2 + \cdots) \frac{\del}{\del y} + (-s - 3b_3t^2 + \cdots) \frac{\del}{\del z} \]
under our chosen trivialization.  The condition on the $m_i$ and $n_i$ for a section as in \eqref{section_mn} to point towards $q(s)$ at $2p$ is that
\[\det \begin{pmatrix} 2t + \cdots & m_0 + m_1t  + \cdots \\ -s + \cdots & n_0 + n_1 t + \cdots \end{pmatrix} = -sm_0 +(2n_0 + sm_1)t + \cdots\]
vanish to second order in $t$.  When $s\neq 0$, this cuts out the $2$-dimensional subspace $m_0= 2n_0 + sm_1 = 0$ in the four dimensional vector space with coordinates $m_0, m_1, n_0, n_1$.

\emph{In characteristic distinct from~$2$}, the limit as $s \to 0$ of this subspace
is simply $m_0 = n_0 = 0$, i.e.\ the subspace $N_D(-p)|_{2p} \subset N_D|_{2p}$
as claimed.
\end{proof}

\begin{cor} \label{cor-dplus}
Suppose that $N_D$ is (semi)stable for $D$ a general Brill--Noether curve
of degree $d$ and genus $g$ in $\pp^3$.
Then $N_C$ is (semi)stable for $C$ a general Brill--Noether curve of degree $d + \epsilon$
and genus $g$ in $\pp^3$, where
\[\epsilon = \begin{cases}
1 & \text{if $\chara(k) \neq 2$;} \\
2 & \text{if $\chara(k) = 2$.}
\end{cases}\]
\end{cor}
\begin{proof}
We specialize $C$ to the union of a
general Brill--Noether curve $D$ with $\epsilon$ one-secant lines.
Applying Lemma~\ref{lem:1_sec_balanced}, it suffices to show that
$N_D[2p \to q]$ (respectively $N_D[2p_1 \to q_1][2p_2 \to q_2]$) is (semi)stable,
where the $p_i$ denote general points on $D$, and the $q_i$ denote general points in $\pp^3$.

As we limit $p_1$ and $p_2$ together to a common point $p$, the
vector bundles $N_D[2p_1 \to q_1][2p_2 \to q_2]$ fit together
to form a vector bundle with central fiber $N_D(-2p)$
(c.f.\ the discussion in Section \ref{sec:mod_in_fam}) --- which is (semi)stable by assumption.

In characteristic distinct from $2$, we apply
Lemma~\ref{lem:1_sec} to conclude that $N_D[2p \to q]$
is (semi)stable as desired.
\end{proof}

\section{Reduction to a finite list of $(d,g)$}\label{sec:reduction_finite}

In this section we combine the results
of the previous section to reduce the proof of Theorem~\ref{thm:main} to a finite list of base cases.

\begin{prop}\label{prop:reduction_finite}
Suppose that Theorem~\ref{thm:main} holds for
curves of degree $d$ and genus $g$ with
\begin{multline}\label{finite_list}
(d,g) \in \{(3, 0), (4, 1), (5, 1), (6,2), (7,2), (6,3), (7,3), \\
(7,4), (8,4), (7,5), (8,5), (8,6), (9,6), (9,7), (10,7), (9,8), (10,8) \}. 
\end{multline}
Then Theorem~\ref{thm:main} holds in all cases.
If the characteristic of the ground field is not $2$, then it suffices to replace list \eqref{finite_list} with
\begin{equation}\label{finite_list_not2}
(d,g) \in \{(3, 0), (4, 1), (6,2), (6,3),  (7,4),  (7,5), (8,6),  (9,7),  (9,8) \}. 
\end{equation}
\end{prop}

\begin{proof}
We will prove this by induction on $d$ and $g$.
By Lemma \ref{lem:pancake}, it suffices to prove this when $g \leq 8$.
If the characteristic is not equal to $2$,
then by Corollary~5.3, it suffices to check (semi)stability for the smallest
degree in each genus
for which Theorem~\ref{thm:main} asserts that the normal bundle
is (semi)stable.
Similarly, if the characteristic is $2$,
it suffices to check (semi)stability for the two smallest degrees.

Note that, for rational curves of even degrees in characteristic $2$,
we have already established
that the normal bundles are unstable.
Thus we do not need to include $(4, 0)$ in our list \eqref{finite_list}.
%
%
\end{proof}

\begin{rem}
By Lemma~\ref{lem:01}, we already know semistability for
$(d, g) = (3, 0)$ and $(4, 1)$.
This establishes Theorem~\ref{thm:main} for curves of genus $0$
in any characteristic, and for curves of genus $1$ in characteristic
distinct from $2$.
\end{rem}

\begin{rem}
The reason that the cases $(6, 2)$ and $(7, 4)$
appeared in our list \eqref{finite_list} of remaining cases
is that the cases $(5, 2)$ and $(6, 4)$
were exceptions to Theorem~\ref{thm:main},
and so our induction on the degree broke down.
In fact, one cannot degenerate such curves to the union
of a Brill--Noether curve $D$ of degree $d - 1$ and genus $g$
with a $1$-secant line and apply Lemma~\ref{lem:1_sec_balanced}
(even without applying Lemma~\ref{lem:1_sec});
in both cases, $N_D[2p \to q]$
is unstable (if $Q$ denotes the unique quadric containing $D$
then $N_{D/Q}(-2p) \subset N_D[2p \to q]$ is destabilizing).
\end{rem}

\section{Base Cases: Applications of Gluing Data}

In this section, we establish those base cases appearing 
in Proposition~\ref{prop:reduction_finite} which can be studied
using the techniques of Section~\ref{sec-gluing}.

\subsection*{\boldmath The case $(d, g) = (5, 1)$}
We degenerate to the union of an elliptic normal curve $C$
with a $1$-secant line.
By Lemma~\ref{lem:1_sec_balanced},
it suffices to show $N_C[2u \to v]$ is semistable, where $u \in C$ and $v \in \pp^3$ are general.
Fix a quadric $Q$ containing $C$,
and specialize $v$ to a general point on $C$.
By Lemma~\ref{lem:3pts}, there are exactly two points on $C$
at which the fibers of $N_{C \to v}$ and $N_{C/Q}$ meet transversely;
specialize $u$ to one of them. Then applying \eqref{modifications_ses} with \(k=1\) (by virtue of Lemma~\ref{lem:3pts}) to the normal bundle exact sequence \eqref{eq:normal_bundle_exact_seq} for \(C \subset Q\), we see that $N_C[2u \to v]$ fits in an exact sequence
\[0 \to [N_{C/Q}(-u) \simeq \O_C(2)(-u)] \to N_C[2u \to v] \to [N_Q|_C(-u) \simeq \O_C(2)(-u)] \to 0,\]
so is semistable as desired.

\subsection*{\boldmath The cases $(d, g) = (9, 7), (10, 7), (9, 8)$, and $(10, 8)$}

When $(d,g) =(10,7)$, or $(10,8)$, respectively, we first degenerate the curve to the union of a general Brill--Noether curve $C$ of degree $9$ and genus $7$ or $8$, respectively, and general $1$-secant line $M$, meeting $C$ at $u$.
Choose a point $v \in M \smallsetminus u$ so $M = \overline{uv}$.
By Lemma \ref{lem:1_sec_balanced}, it suffices to show that $N_C(u)[2u \to v]$ is stable.

Therefore, in order to deal with all of our cases $(d, g) \in \{(9, 7), (10,7), (9, 8), (10,8)\}$, we  begin with a curve $C$ of degree $9$ and genus $7$ or $8$.  We will degenerate $C$
to the union of a general canonical curve $D$ (of degree $6$ and genus $4$)
and a union $R$ of rational curves meeting $D$ quasi-transversely at a set $\Gamma$ of $6$ points
(three general $2$-secant lines when $g=7$, and the union of 
a general $2$-secant line with a general $4$-secant conic when $g=8$, respectively).  
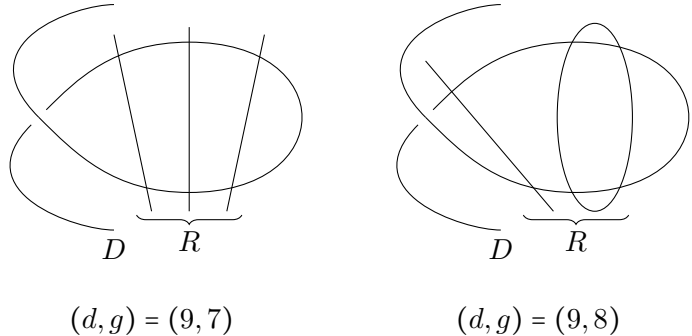
\begin{center}
\begin{tikzpicture}
\draw (1, 3.5) .. controls (0.5, 3.5) and (-1, 3) .. (0, 2);
\draw (0, 2) .. controls (0.5, 1.5) and (1, 1) .. (2, 1);
\draw (0.1, 2.1) .. controls (0.5, 2.5) and (1, 3) .. (2, 3);
\draw (1, 0.5) .. controls (0.5, 0.5) and (-1, 1) .. (-0.1, 1.9);
\draw (2, 1) .. controls (4, 1) and (4, 3) .. (2, 3);
\draw (1,3.1)--(1.5,.75);
\draw (2,3.2)--(2,.75);
\draw (3,3.1)--(2.5,.75);
\draw (1, 0.25) node{$D$};
\draw [decorate, decoration={brace, mirror, amplitude=0.75ex}] (1.3, 0.7) -- (2.7, 0.7);
\draw (2, 0.35) node{$R$};
\draw (1.5, -0.7) node{$(d, g) = (9, 7)$};
\end{tikzpicture}
\begin{tikzpicture}
\draw (1, 3.5) .. controls (0.5, 3.5) and (-1, 3) .. (0, 2);
\draw (0, 2) .. controls (0.5, 1.5) and (1, 1) .. (2, 1);
\draw (0.1, 2.1) .. controls (0.5, 2.5) and (1, 3) .. (2, 3);
\draw (1, 0.5) .. controls (0.5, 0.5) and (-1, 1) .. (-0.1, 1.9);
\draw (2, 1) .. controls (4, 1) and (4, 3) .. (2, 3);
\draw (2.25, 2) ellipse (0.5 and 1.25);
\draw (0,2.75)--(1.7,.75);
\draw (1, 0.25) node{$D$};
\draw [decorate, decoration={brace, mirror, amplitude=0.75ex}] (1.3, 0.7) -- (2.7, 0.7);
\draw (2, 0.35) node{$R$};
\draw (1.5, -0.7) node{$(d, g) = (9, 8)$};
\end{tikzpicture}
\end{center}

Write $Q$ for the unique quadric containing $D$.
In both cases, the tangent lines to $R$ at $\Gamma$ are transverse to $Q$,
and so applying \eqref{modifications_ses} with \(k=0\) to the normal bundle exact sequence \eqref{eq:normal_bundle_exact_seq} for \(D \subset Q\), we see that the restricted normal bundle $N_{D \cup R}|_D$ fits into a balanced exact sequence:
\begin{equation} \label{dseq}
0 \to [N_{D/Q} \simeq \O_D(3)] \to N_{D \cup R}|_D \to [N_Q|_D(\Gamma) \simeq \O_D(2)(\Gamma)] \to 0.
\end{equation}
In particular, $N_{D \cup R}|_D$ is strictly semistable,
and $N_{D/Q}$ gives a destabilizing line bundle.

Similarly, after specializing $v$ to a point on $D$, Lemma \ref{lem:3pts} asserts that there are $4$ points $u$ on $D$ where the fibers $N_{D \to v}|_u$ and $N_{D/Q}|_u$ coincide to first order.  Specializing $u$ to one of these points, and  applying \eqref{modifications_ses} with \(k=1\), we again have a balanced exact sequence
\begin{equation} \label{dseq_2}
0 \to N_{D/Q} \to N_{D \cup R}|_D(u)[2u \to v] \to N_Q|_D(\Gamma) \to 0.
\end{equation}
In particular, $N_{D \cup R}|_D(u)[2u \to v]$ is strictly semistable,
and $N_{D/Q}$ gives a destabilizing line bundle.

Let $L$ be a line component of $R$,
meeting $D$ at $p_1$ and $p_2$ with 
 $p_i' \in T_{p_i}D \smallsetminus p_i$, and denote by
 $\Lambda_i$ the plane spanned by $p_i'$ and $L$.  Then
 \[N_{D \cup R}|_L \simeq N_{L / \Lambda_1}(p_1) \oplus N_{L / \Lambda_2}(p_2) \simeq \O_{\pp^1}(2) \oplus \O_{\pp^1}(2). \] 
Combining this with Lemma~\ref{lem:4_sec_conic},
the restriction of $N_{D \cup R}$ (resp.~$N_{D \cup R}(u)[2u \to v]$) to
each of the components of $R$ is also strictly semistable.

In particular, writing $\nu \colon D \sqcup R \to D \cup R$
for the normalization, any destabilizing subbundle $F \subset \nu^* N_{D \cup R}$ (resp.~$F \subset \nu^* N_{D \cup R}(u)[2u \to v]$)
must be destabilizing
on every component and agree at the points lying over the nodes $D \cap R$.
The key observation is that, because $N_{D/Q}$ is a subbundle of $N_D$ as well,
its fiber 
at each of the points of $\Gamma$ is exactly the subspace that does not smooth that node.
On the other hand, if $L$ denotes a component of $R$ which is a line,
then any destabilizing $\O(2)$
has a fiber at one or more of the nodes that fails to smooth it
(otherwise it would be a subbundle of $N_L \simeq \O(1) \oplus \O(1)$).
It thus remains to show that $N_{D/Q}$ is the \emph{unique} destabilizing
subbundle of $N_{D \cup R}|_D$ (resp.~$N_{D \cup R}|_D(u)[2u \to v]$), or equivalently:

\begin{lem}
The sequences  \eqref{dseq} and \eqref{dseq_2} are nonsplit, i.e.
\[H^0(N_{D \cup R}|_D(-2)(-\Gamma)) = 0 \qquad\text{and} \qquad H^0(N_{D \cup R}|_D(-2)(-\Gamma)(u)[2u \to v]) = 0.\]
\end{lem}
\begin{proof}
To show the desired vanishing,
we degenerate
two points of $\Gamma$ together to a common point $p$ on $D$:
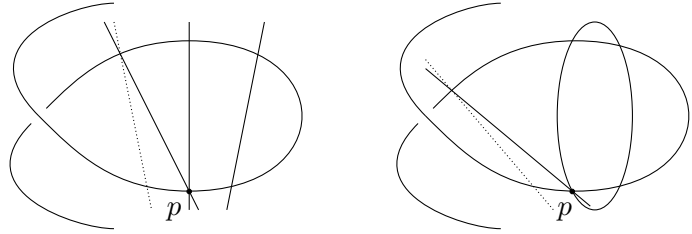
\begin{center}
\begin{tikzpicture}
\draw (1, 3.5) .. controls (0.5, 3.5) and (-1, 3) .. (0, 2);
\draw (0, 2) .. controls (0.5, 1.5) and (1, 1) .. (2, 1);
\draw (0.1, 2.1) .. controls (0.5, 2.5) and (1, 3) .. (2, 3);
\draw (1, 0.5) .. controls (0.5, 0.5) and (-1, 1) .. (-0.1, 1.9);
\draw (2, 1) .. controls (4, 1) and (4, 3) .. (2, 3);
\draw[densely dotted] (1,3.25)--(1.5,.75);
\draw (0.87,3.25)--(2.125,.75);
\draw (2,3.25)--(2,.75);
\draw (3,3.25)--(2.5,.75);
\filldraw (2,.995) circle[radius=0.03];
\draw (1.8,.98) node[below]{$p$};
\end{tikzpicture}
\begin{tikzpicture}
\draw (1, 3.5) .. controls (0.5, 3.5) and (-1, 3) .. (0, 2);
\draw (0, 2) .. controls (0.5, 1.5) and (1, 1) .. (2, 1);
\draw (0.1, 2.1) .. controls (0.5, 2.5) and (1, 3) .. (2, 3);
\draw (1, 0.5) .. controls (0.5, 0.5) and (-1, 1) .. (-0.1, 1.9);
\draw (2, 1) .. controls (4, 1) and (4, 3) .. (2, 3);
\draw (2.25, 2) ellipse (0.5 and 1.25);
\draw[densely dotted] (0,2.75)--(1.7,.75);
\draw (0,2.63)--(2.19,.8);
\filldraw (1.95,.995) circle[radius=0.03];
\draw (1.85,.98) node[below]{$p$};
\end{tikzpicture}
\end{center}
Let $N$ denote the bundle obtained by gluing $N_{D \cup R}|_{D \smallsetminus p}$
to $N_D(p)|_{D \smallsetminus (\Gamma \smallsetminus p)}$ along the natural isomorphism
$N_{D \cup R}|_{D \smallsetminus \Gamma} \simeq N_D(p)|_{D \smallsetminus \Gamma}$.
By Section \ref{sec:mod_in_fam}, the bundles $N_{D \cup R}|_D$ (resp.~$N_{D \cup R}|_D(u)[2u \to v]$) fit together to form a bundle whose
central fiber is the bundle $N$ (resp.~$N(u)[2u \to v]$). 
It thus remains to show
\[H^0(N(-2)(-\Gamma)) = 0 \qquad\text{and} \qquad H^0(N(u)[2u \to v](-2)(-\Gamma)) = 0.\]
To do this, we use the exact sequences coming from applying \eqref{modifications_ses} (with \(k=1\) for the modification at \(u\) in the second case) to the normal bundle sequence for \(D \subset Q\):
\begin{gather*}
0 \to [N_{D/Q}(p) \simeq \O_D(3)(p)] \to N \to [N_Q|_D(\Gamma - p) \simeq \O_D(2)(\Gamma - p)] \to 0 \\
0 \to [N_{D/Q}(p) \simeq \O_D(3)(p)] \to N(u)[2u \to v] \to [N_Q|_D(\Gamma - p) \simeq \O_D(2)(\Gamma - p)] \to 0;
\end{gather*}
twisting these sequences by $\O_D(-2)(-\Gamma)$ and taking global sections, it remains to check that
\[H^0(\O_D(1)(-(\Gamma - p))) = H^0(\O_D(-p)) = 0.\]
This is clear since the five points of $\Gamma - p = \Gamma^{\text{red}}$
are in linear general position.
\end{proof}

\section{Stability and degeneration III: Limits of Gluing Data}\label{sec:limit_gluing}

As in the previous section, 
we again want to degenerate to reducible curves $X \cup Y$ where neither $N_{X \cup Y}|_X$ nor $N_{X\cup Y}|_Y$ are necessarily stable, but the destabilizing subbundles on each component do not agree at $X\cap Y$.
The fundamental difficulty we address in this section is that it is often difficult to compute the destabilizing subbundles on each component without further degeneration.
We therefore study the agreement conditions at $X \cap Y$ as the points of $X \cap Y$ come together.

Let $D$ be a Brill--Noether curve.  Fix distinct points $q, p_{11}, \ldots, p_{1r_1}, p_{21}, \ldots, p_{2r_2} \in D$.
Let $R_i$ be a rational curve meeting $D$ quasi-transversely exactly at $q, p_{i1}, \ldots, p_{ir_i}$,
such that the tangent directions at $q$ to $D$, $R_1$, and $R_2$ span $\pp^3$.

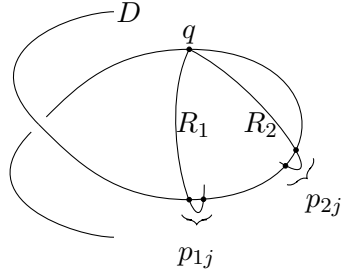
\begin{center}
\begin{tikzpicture}
\draw (1, 3.5) .. controls (0.5, 3.5) and (-1, 3) .. (0, 2);
\draw (0, 2) .. controls (0.5, 1.5) and (1, 1) .. (2, 1);
\draw (0.1, 2.1) .. controls (0.5, 2.5) and (1, 3) .. (2, 3);
\draw (1, 0.5) .. controls (0.5, 0.5) and (-1, 1) .. (-0.1, 1.9);
\draw (2, 1) .. controls (4, 1) and (4, 3) .. (2, 3);
\draw (2, 3) .. controls (1.5, 2) and (2.2, 0) .. (2.2, 1.2);
\draw (2, 3) .. controls (3, 2.5) and (4, 1) .. (3.2, 1.5);
\draw (2.05, 2) node{$R_1$};
\draw (2.95, 2) node{$R_2$};
\draw (1.2, 3.5) node{$D$};
\draw (2, 3.2) node{$q$};
\filldraw (2, 3) circle [radius=0.03];
\filldraw (2, 1) circle [radius=0.03];
\filldraw (2.19, 1.005) circle [radius=0.03];
\filldraw (3.28, 1.45) circle [radius=0.03];
\filldraw (3.42, 1.66) circle [radius=0.03];
\draw [decorate, decoration={brace, mirror, amplitude=0.8ex}] (1.9, 0.8) -- (2.3, 0.8);
\draw (2.1, 0.5) node[below]{$p_{1j}$};
\draw [decorate, decoration={brace, mirror, amplitude=0.8ex}] (3.3, 1.2) -- (3.65, 1.55);
\draw (3.8, 1.2) node[below]{$p_{2j}$};
\end{tikzpicture}
\end{center}

Assume that both $R_i$ satisfy Assumption \ref{assump:balanced_slope}.  
Using this assumption we may apply Lemma \ref{lem:deform_delta} to show that there exists an \'etale neighborhood $\Delta = q_i(t)$ of $q \in D$,
which we normalize so $q_1(t)$ and $q_2(t)$ have distinct derivatives at $t = 0$,
and deformations $R_i(t)$ of $R_i$, and $p_{ij}(t)$ of $p_{ij}$,
such that for $t \in \Delta$, the rational curve $R_i(t)$ meets $D$ quasi-transversely in $q_i(t), p_{i1}(t), \dots p_{ir_i}(t)$.

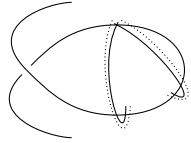
\begin{center}
\begin{tikzpicture}[scale=0.6]
\draw (1, 3.5) .. controls (0.5, 3.5) and (-1, 3) .. (0, 2);
\draw (0, 2) .. controls (0.5, 1.5) and (1, 1) .. (2, 1);
\draw (0.1, 2.1) .. controls (0.5, 2.5) and (1, 3) .. (2, 3);
\draw (1, 0.5) .. controls (0.5, 0.5) and (-1, 1) .. (-0.1, 1.9);
\draw (2, 1) .. controls (4, 1) and (4, 3) .. (2, 3);
\draw (2, 3) .. controls (1.5, 2) and (2.2, 0) .. (2.2, 1.2);
\draw (2, 3) -- (2.03, 3.06);
\draw (2, 3) -- (1.94, 3.03);
\draw (2, 3) .. controls (3, 2.5) and (4, 1) .. (3.2, 1.5);
\draw[densely dotted] (1.9, 3.1) .. controls (1.4, 2) and (2.2, -0.3) .. (2.3, 1.2);
\draw[densely dotted] (2.1, 3.1) .. controls (3, 2.7) and (4.3, 0.9) .. (3.1, 1.45);
\end{tikzpicture}
\end{center}
Suppose that, for $t \in \Delta^* \colonequals \Delta\smallsetminus 0$,
the normal bundle $N_{D \cup R_1(t) \cup R_2(t)}$ is not stable.
These bundles fit together to form a vector bundle $\hat{\N}$ over $\Delta^*$.
However, since $D \cup R_1 \cup R_2$ is not lci,
its normal sheaf is not a vector bundle; there is therefore no obvious way to extend $\hat{\N}$
over $\Delta$.
Thus, extracting information at the central fiber is subtle.

By the discussion in Section \ref{sec:mod_in_fam},
we may nevertheless extend the \emph{restriction} $\hat{\N}|_D$
to a bundle $\N$ on $D \times \Delta$
whose fiber $N \colonequals \N|_0$ over $0 \in \Delta$ is obtained from gluing
$N_{D \cup R_1 \cup R_2}|_{D \smallsetminus q}$
to $N_D(q)|_{D \smallsetminus \{p_{11}, \ldots, p_{1r_1}, p_{21}, \ldots, p_{2r_2}\}}$
along the natural isomorphism
\[N_{D \cup R_1 \cup R_2}|_{D \smallsetminus \{q, p_{11}, \ldots, p_{1r_1}, p_{21}, \ldots, p_{2r_2}\}} \simeq N_D|_{D \smallsetminus \{q, p_{11}, \ldots, p_{1r_1}, p_{21}, \ldots, p_{2r_2}\}} \simeq N_D(q)|_{D \smallsetminus \{q, p_{11}, \ldots, p_{1r_1}, p_{21}, \ldots, p_{2r_2}\}}.\]

Write $\nu \colon D \sqcup R_1(t) \sqcup R_2(t) \to D \cup R_1(t) \cup R_2(t)$
for the normalization map.
Let $\hat{\L} \subset \nu^* \hat{\N}$
be a destabilizing line bundle, i.e.\ which satisfies $\mu^{\text{adj}}(\hat{\L}) \geq \mu(\hat{\N})$.
Let $\ell_D$, $\ell_1$, and $\ell_2$
denote the slopes of the restriction of $\hat{\L}$ to $D$, $R_1(t)$, and $R_2(t)$,
and $c$ denote the number of nodes of $D \cup R_1(t) \cup R_2(t)$
above which the fibers of $\hat{\L}$ do not coincide
(for $t \in \Delta^*$).
Since being perfectly balanced is open, Condition~\ref{assump:balanced_slope}
implies that the $\hat{\N}|_{R_i(t)}$ are perfectly balanced.
We therefore have
\begin{equation} \label{ell12}
\ell_i \leq \mu(\hat{\N}|_{R_i(t)}) \quad \text{and} \quad c \geq 0,
\end{equation}
but
\[\mu^{\text{adj}}(\hat{\L}) = \ell_1 + \ell_2 + \ell_D - c \geq \mu(\hat{\N}|_{R_1(t)}) + \mu(\hat{\N}|_{R_2(t)}) + \mu(\hat{\N}|_D).\]
If $\ell_D > \mu(\hat{\N}|_D)$, i.e.\ $\N^* = \hat{\N}|_D$ is unstable, then
$N$ is unstable by Proposition~\ref{prop:stab-open}.
Thus either:
\begin{enumerate}[(i)]
\item \label{unstable} $N$ is unstable, or
\item \label{semistable} \eqref{ell12} is an equality --- i.e.\ $\ell_i = \mu(\hat{\N}|_{R_i(t)})$ and $c = 0$ --- and $\ell_D = \mu(N)$.
\end{enumerate}

In case~(\ref{semistable}), our first task is to translate the condition that \eqref{ell12} is an equality
to information about the restriction $\L^* = \hat{\L}|_D$.
(The condition that $\ell_D = \mu(N)$ already concerns $\L^*$.)
To do this, observe that since the $\hat{\N}|_{R_i(t)}$
are perfectly balanced, we have a canonical isomorphism
\[\varphi_{ij}^* \colon\ \pp \N^*|_{q_i(t)} \xlongrightarrow{\sim} \pp \N^*|_{p_{ij}(t)} \quad \text{for $t \in \Delta^*$}.\]
Writing $\L^* = \hat{\L}|_D$, the condition that \eqref{ell12} is an equality then implies that
\begin{equation} \label{translate}
\L^*|_{p_{ij}(t)} = \varphi_{ij}^*(\L^*|_{q_i(t)}) \quad \text{for $t \in \Delta^*$}.
\end{equation}

By Proposition~\ref{prop:stab-open},
we can extend $\L^*$ across the central fiber to a subbundle $\L \subset \N$,
and consider the restriction $L \colonequals \L|_0 \subset N$ to the central fiber. Our second task is to figure out what \eqref{translate} implies
for $L$. (Figuring out what $\ell_D = \mu(N)$ implies for $L$ is easy:
Since $\mu$ is constant in flat families, it implies $\mu(L) = \mu(N)$.)

To do this, we
observe that the bundles $N_{D \cup R_i(t)}$ fit together
to form bundles $\hat{\N}_i$ over $\Delta$ (including over $t = 0$).
Writing $\N_i = \hat{\N}_i|_D$,
there are natural inclusions
$\N_i \subset \N$,
which are isomorphisms away from $R_{\bar{i}}(t) \cap D$ (here $\bar{i} = 3 - i$ denotes the other index) --- 
so in particular at $q_i(t)$ for $t \neq 0$, and at $p_{ij}(t)$ for all $t$.
This inclusion induces a birational isomorphism on projectivizations
$\pp \N_i \dashrightarrow \pp \N$.
The advantage to working with $\N_i$ is that
$\hat{\N}_i|_{R_i(t)}$ is perfectly balanced,
so we obtain
regular maps \emph{defined over $\Delta$ (in particular for $t = 0$)}:
\[\varphi_{ij} \colon\ \pp \N_i|_{q_i(t)} \xlongrightarrow{\sim} \pp \N_i|_{p_{ij}(t)} \quad \text{for $t \in \Delta$},\]
that are compatible with the $\varphi_{ij}^*$ in the sense that the following diagram commutes:
\[\begin{tikzcd}
\pp \N_i|_{q_i(t)} \arrow[r, "\varphi_{ij}"] \arrow[d, dashed] & \pp \N_i|_{p_{ij}(t)} \arrow[d, equal]\\
\pp \N|_{q_i(t)} \arrow[r, "\varphi_{ij}^*", dashed] & \pp \N|_{p_{ij}(t)}
\end{tikzcd}\]

We now restrict to the graph of $q_i(t)$.
Then the map $\N_i \subset \N$ drops rank
exactly over $t = 0$. Its kernel at $t = 0$ is
the one-dimensional subspace $D_i \subset N_{D \cup R_i}|_q$
corresponding to sections that fail to smooth the node at $q$,
and its image is given by the one-dimensional
subspace of $F_i \subset N|_q$ corresponding to the tangent direction of $R_i$ at $q$.
The rational map $\pp\N_i \dashrightarrow \pp\N$
is thus obtained by blowing up at $D_i$,
and contracting the proper transform of the fiber over $q$ to $F_i$:
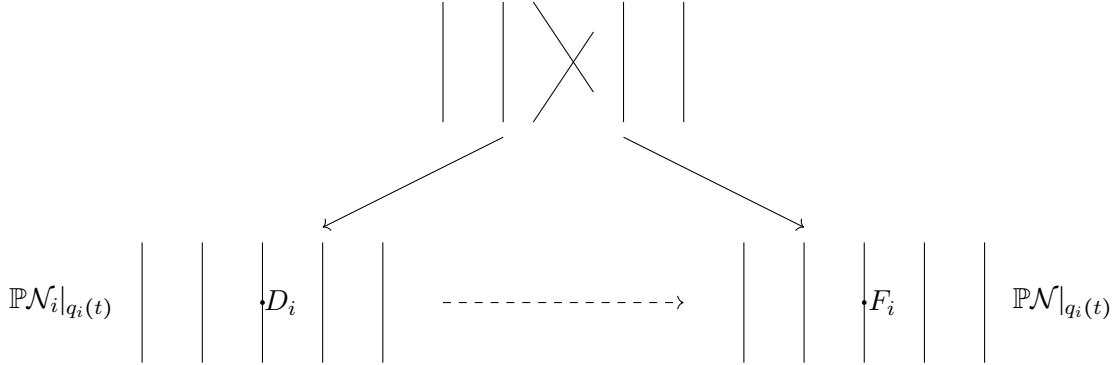
\begin{center}
\begin{tikzpicture}[scale=0.8]
\draw (5, 4) -- (5, 6);
\draw (6, 4) -- (6, 6);
\draw (6.5, 4) -- (7.5, 5.5);
\draw (6.5, 6) -- (7.5, 4.5);
\draw (8, 4) -- (8, 6);
\draw (9, 4) -- (9, 6);
\draw (0, 0) -- (0, 2);
\draw (1, 0) -- (1, 2);
\draw (2, 0) -- (2, 2);
\draw (3, 0) -- (3, 2);
\draw (4, 0) -- (4, 2);
\draw (10, 0) -- (10, 2);
\draw (11, 0) -- (11, 2);
\draw (12, 0) -- (12, 2);
\draw (13, 0) -- (13, 2);
\draw (14, 0) -- (14, 2);
\draw[dashed, ->] (5, 1) -- (9, 1);
\draw[->] (6, 3.75) -- (3, 2.25);
\draw[->] (8, 3.75) -- (11, 2.25);
\filldraw (2, 1) circle[radius=0.03];
\filldraw (12, 1) circle[radius=0.03];
\draw(-0.3,1) node[left]{$\pp \N_i|_{q_i(t)}$};
\draw (2.3, 1) node{$D_i$};
\draw (12.3, 1) node{$F_i$};
\draw(14.3,1) node[right]{$\pp \N|_{q_i(t)}$};
\end{tikzpicture}
\end{center}

The line subbundle $\L|_{q_i(t)} \subset \N|_{q_i(t)}$ defines a section of $\pp \N|_{q_i(t)}$ and (by curve-to-projective extension) of $\pp \N_i|_{q_i(t)}$; if the first of these sections does not pass through $F_i$, then the second \emph{must} pass through $D_i$.
Combining this with \eqref{translate}, when we pass to the central
fiber, the fibers of $L$ at the $p_{ij}$
can sometimes be described in terms of
\[D_{ij} \colonequals \varphi_{ij}(D_i).\]

Namely, by our assumption that the tangent directions to $D$, $R_1$, and $R_2$
span $\pp^3$, the subspaces $F_1$ and $F_2$ are disjoint.
The fiber $L|_q \subset N|_q$ thus either:
\begin{enumerate}[(a)]
\item\label{fibera} Coincides with neither $F_1$ nor $F_2$:
In this case, $L|_{p_{ij}} = D_{ij}$.
\item\label{fiberb} Coincides with $F_1$ but not $F_2$:
In this case, $L|_{p_{2j}} = D_{2j}$
and $L|_q = F_1$.
\item\label{fiberc} Coincides with $F_2$ but not $F_1$:
In this case, $L|_{p_{1j}} = D_{1j}$
and $L|_q = F_2$.
\end{enumerate}
The upshot of this is the following lemma.

\begin{lem}\label{lem:limit_gluing} With the above notation, if
\[\begin{array}{ll}
\text{every sub-line-bundle of\ldots} & \text{has slope\ldots} \\
N & \leq \mu(N) \\
N[p_{ij} \to D_{ij}] & < \mu(N) \\
N[q \to F_1][p_{2j} \to D_{2j}] & < \mu(N) \\
N[q \to F_2][p_{1j} \to D_{1j}] & < \mu(N),
\end{array}\]
then $N_{D \cup R_1(t) \cup R_2(t)}$ is stable, for $t \in \Delta$ generic.
In particular, if these four vector bundles are merely
\emph{semistable}, then $N_{D \cup R_1(t) \cup R_2(t)}$ is \emph{stable} for $t \in \Delta$
generic.
\end{lem}

Now suppose that $R_i$ is a $2$-secant line (meeting $D$ at $q$ and $p_{i1}$),
and write $q' \in T_{q}D \smallsetminus q$ and $p_{i1}' \in T_{p_{i1}}D \smallsetminus p_{i1}$
for points on the tangent lines to $D$ at $q$ and $p_{i1}$ respectively.
Then we have the explicit decomposition (c.f., Example \ref{ex:line})
\begin{equation}\label{eq:2sec_decomp}N_{D \cup R_i}|_{R_i} \simeq N_{R_i \to q'}(q) \oplus N_{R_i \to p_{i1}'}(p_{i1}) \simeq \O_{\pp^1}(2)^{\oplus 2}.\end{equation}
In particular, we see that Assumption~\ref{assump:balanced_slope} is satisfied.
Moreover, we may use this decomposition to compute
the subspace $D_{i1}$: In terms of \eqref{eq:2sec_decomp}, 
\[D_i = N_{R_i \to p_{i1}'}(p_{i1})|_q \quad \Rightarrow \quad D_{i1} = N_{R_i \to p_{i1}'}(p_{i1})|_{p_{i1}}. \]
To describe this in a way that is compatible with the isomorphism
\[N_{D \cup R_i}|_D \simeq N_D(q + p_{i1})[q \to p_{i1}][p_{i1} \to q],\]
we apply Lemma~8.4 of \cite{aly}, which states that under this isomorphism we have
\begin{equation} \label{Di}
D_{i1} = N_{D \to q}(p_{i1})|_{p_{i1}} \subset N_D(q + p_{i1})[q \to p_{i1}][p_{i1} \to q]|_{p_{i1}}.
\end{equation}
When both $R_1$ and $R_2$ are $2$-secant lines, we have \(N \simeq N_D[p_{11} \to q][p_{21} \to q]\).
Substituting in the \(D_{i1}\) given in \eqref{Di},
Lemma~\ref{lem:limit_gluing} thus gives:

\begin{cor}\label{cor:limit_gluing_22sec}
If $R_1$ and $R_2$ are $2$-secant lines, and the
bundles
\begin{enumerate}[(a)]
\item\label{22sec_a} $N_D[p_{11} \to q][p_{21} \to q]$,
\item\label{22sec_b} $N_D[2p_{11} \to q][2p_{21} \to q]$,
\item\label{22sec_d} $N_D[p_{11} \to q][q \to p_{11}][2p_{21} \to q]$, and
\item\label{22sec_c} $N_D[2p_{11} \to q][p_{21} \to q][q \to p_{21}]$.
\end{enumerate}
are all \emph{semistable}, then $N_{D \cup R_1(t) \cup R_2(t)}$ is \emph{stable} for $t \in \Delta$ generic.
\end{cor}

\begin{rem}
Since \eqref{22sec_c} is obtained from \eqref{22sec_d} by permuting $p_{21}$ and $p_{11}$,
it suffices to prove semistability of \eqref{22sec_a}--\eqref{22sec_d}.
\end{rem}

Now suppose only that $R_1$ is a $2$-secant line.
Applying Lemma~\ref{lem:limit_gluing}, the stability of $N_{D \cup R_1(t) \cup R_2(t)}$ for $t \in \Delta$ generic follows
from the assertions that:
\[\begin{array}{ll}
\text{every sub-line-bundle of\ldots} & \text{has slope\ldots} \\
N & \leq \mu(N) \\
N[p_{11} \to q][p_{2j} \to D_{2j}] & < \mu(N) \\
N[q \to p_{11}][p_{2j} \to D_{2j}] & < \mu(N) \\
N[q \to F_2][p_{11} \to q] & < \mu(N).
\end{array}\]
This follows in turn from the assertion that
\[N[p_{11} \to q] \quad \text{and} \quad N[q \to p_{11}]\]
are stable. We therefore have:

\begin{cor}\label{cor:limit_gluing_4sec2sec}
Suppose that $R_1$ is a $2$-secant line, and write $p_{2j}' \in T_{p_{2j}} R_2 \smallsetminus p_{2j}$
for points on the tangent lines to $R_2$ at the $p_{2j}$.
If the bundles
\begin{enumerate}[(a)]
\item $N_D[p_{2j} \to p_{2j}'][2p_{11} \to q]$ and
\item $N_D[p_{2j} \to p_{2j}'][p_{11} \to q][q \to p_{11}]$
\end{enumerate}
are both \emph{stable/semistable}, then $N_{D \cup R_1(t) \cup R_2(t)}$ is \emph{stable} for $t \in \Delta$ generic.
\end{cor}

The bundles $N_D[p_{2j} \to p_{2j}'][2p_{11} \to q]$ and $N_D[p_{2j} \to p_{2j}'][p_{11} \to q][q \to p_{11}]$
appearing in Corollary~\ref{cor:limit_gluing_4sec2sec} are rank~$2$ vector bundles
of odd degree, and hence stability is equivalent to semistability.


\section{Base Cases: Applications of Limits of Gluing Data}

\subsection*{\boldmath The cases $(d, g) = (7, 2), (6, 3), (7, 3), (7, 4), (8, 4)$, and $(8, 5)$}
In these cases, we degenerate to the union of a general Brill--Noether $D$ curve of degree $d - 2$ and genus $g - 2$, a $2$-secant line $R_1$ through general points $q$ and $p_{11}$, and a $2$-secant line $R_2$ through $q$ and another general point $p_{21}$. 
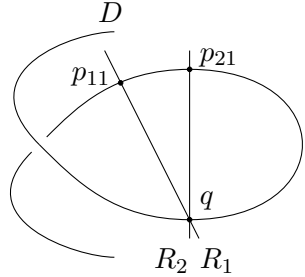
\begin{center}
\begin{tikzpicture}
\draw (1, 3.5) .. controls (0.5, 3.5) and (-1, 3) .. (0, 2);
\draw (0, 2) .. controls (0.5, 1.5) and (1, 1) .. (2, 1);
\draw (0.1, 2.1) .. controls (0.5, 2.5) and (1, 3) .. (2, 3);
\draw (1, 0.5) .. controls (0.5, 0.5) and (-1, 1) .. (-0.1, 1.9);
\draw (2, 1) .. controls (4, 1) and (4, 3) .. (2, 3);
\draw (0.87,3.25)--(2.125,.75);
\draw (2,3.25)--(2,.75);
\draw (2,.75) node[below right]{$R_1$};
\draw (2.125,.75)node[below left]{$R_2$};
\filldraw (2,1) circle[radius=0.03];
\draw (2,1) node[above right ]{$q$};
\filldraw (2,3) circle[radius=0.03];
\draw (2,2.9) node[above right]{$p_{21}$};
\filldraw (1.083,2.82) circle[radius=0.03];
\draw (1.08,2.9) node[left]{$p_{11}$};
\draw (1.25,3.5) node[above left]{$D$};
\end{tikzpicture}
\end{center}

Then $R_1$ and $R_2$ satisfy Assumption \ref{assump:balanced_slope}, and so by Lemma \ref{lem:deform_delta}, the union $D \cup R_1 \cup R_2$ deforms to the union of $D$ and two general $2$-secant lines, which by Lemma \ref{lem:rightcomp}\eqref{lem:rightcomp_2sec} is a Brill--Noether curve of degree $d$ and genus $g$.  By Corollary \ref{cor:limit_gluing_22sec}, it suffices to check that the three bundles \ref{cor:limit_gluing_22sec}\eqref{22sec_a}-\eqref{22sec_d} 
are semistable when $D$ is a general curve of degree $d - 2$ and genus $g - 2$.

\subsection*{\boldmath $(d, g) = (7, 2)$}
Here $D$ is of degree $5$ and genus $0$.
We further degenerate $D$ to the union of a general rational normal curve $C$ (i.e.,~degree $3$ and genus $0$) and two general $1$-secant lines $\bar{u_1, v_1}$ and $\bar{u_2, v_2}$ meeting $C$ at $u_1$ and $u_2$ respectively.  By Lemma \ref{lem:1_sec_balanced}, it therefore suffices to show that the bundles 
\begin{enumerate}[(a)]
\item $N_C[p_{11} \to q][p_{21} \to q][2u_1 \to v_1][2u_2 \to v_2]$, and
\item $N_C[2p_{11} \to q][2p_{21} \to q][2u_1 \to v_1][2u_2 \to v_2]$, and
\item $N_C[p_{11} \to q][q \to p_{11}][2p_{21} \to q][2u_1 \to v_1][2u_2 \to v_2]$,
\end{enumerate}
are semistable.   Limiting $u_1$ to $p_{11}$ and $u_2$ to $p_{21}$ (c.f.\ the discussion in Section \ref{sec:mod_in_fam}), we obtain
\begin{enumerate}[(a)]
\item\label{case_a} $N_C(-p_{11}-p_{21})[p_{11} \to v_1][p_{21} \to v_2]$
\item\label{case_b} $N_C(-2p_{11}-2p_{21})$
\item\label{case_c} $N_C(-p_{11} - 2p_{21})[p_{11} \to v_1][q \to p_{11}]$
\end{enumerate}
After further limiting $p_{11}$ to $p_{21}$ in \eqref{case_a} (resp.~$q$ to $p_{11}$ in \eqref{case_c}), and using the fact that $N_{C \to v_1}|_{p_{11}}$ is a general subspace, these bundles all specialize to twists of $N_C$, and are therefore semistable.

\subsection*{\boldmath $(d, g) = (6, 3)$ and $(7, 3)$}
When $(d, g) = (6, 3)$, then $D$ is of degree $4$ and genus $1$.
For uniformity of notation, we write $C = D$.

When $(d, g) = (7, 3)$, then $D$ is of degree $5$ and genus $1$.
We further degenerate $D$ to the union of a general Brill--Noether
curve $C$ of degree $4$ and genus $1$,
with a general $1$-secant line $M$ meeting $C$ at $u$.
Write $v \in M \smallsetminus u$ for another point on $M$.
By Lemma \ref{lem:1_sec_balanced}, in these cases it suffices
to prove semistability of the bundles \ref{cor:limit_gluing_22sec}\eqref{22sec_a}-\eqref{22sec_d}
with the extra modification $[2u \to v]$.

Combining these cases, it suffices to show that the following $6$ bundles
on $C$ are semistable:
\begin{enumerate}[(a)]
\item \label{g3a} $N_C[p_{11} \to q][p_{21} \to q]$ and $N_C[2u \to v][p_{11} \to q][p_{21} \to q]$,
\item \label{g3b} $N_C[2p_{11} \to q][2p_{21} \to q]$ and $N_C[2u \to v][2p_{11} \to q][2p_{21} \to q]$,
\item \label{g3c} $N_C[p_{11} \to q][q \to p_{11}][2p_{21} \to q]$ and $N_C[2u \to v][p_{11} \to q][q \to p_{11}][2p_{21} \to q]$.
\end{enumerate}

\begin{lem} \label{lem:disaster}
Let $C$ be an irreducible curve, and $u, v, p_{11}, p_{21}, q$ be general points on $C$.
Suppose that the following bundles are semistable:
\begin{enumerate}
\item \label{one} $N_C[2p_{11} \to q]$
\item \label{two} $N_C[2p_{11} \to q][2p_{21} \to q]$
\item \label{three} $N_C[p_{11} \to q][q \to p_{11}][2p_{21} \to q]$.
\end{enumerate}
Then all of the following bundles are also semistable:
\begin{enumerate}[(a)]
\item \label{aa} $N_C[p_{11} \to q][p_{21} \to q]$
\item \label{bb} $N_C[p_{11} \to q][p_{21} \to v]$
\item \label{cc} $N_C[2u \to v][p_{11} \to q][p_{21} \to q]$
\item \label{dd} $N_C[2u \to v][2p_{11} \to q][2p_{21} \to q]$
\item \label{ee} $N_C[2u \to v][p_{11} \to q][q \to p_{11}][2p_{21} \to q]$
\item \label{ff} $N_C[u \to v][v \to u][p_{11} \to q][p_{21} \to q]$
\item \label{gg} $N_C[u \to v][v \to u][2p_{11} \to q][2p_{21} \to q]$
\item \label{hh} $N_C[u \to v][v \to u][2p_{11} \to q][p_{21} \to q][q \to p_{21}]$.
\end{enumerate}
\end{lem}
\begin{proof} We argue by specializing the various points on $C$,
to reduce to twists of bundles that we already assumed or proved were semistable.

\begin{enumerate}[(a)]
\item Specialize $p_{21}$ to $p_{11}$; the resulting bundle is $N_C[2p_{11} \to q]$, i.e.\ \eqref{one}.
\item Specialize $v$ to $q$; the resulting bundle is $N_C[p_{11} \to q][p_{21} \to q]$, i.e.\ \eqref{aa}.
\item Specialize $u$ to $p_{21}$; the resulting bundle is $N_C[p_{11} \to q][p_{21} \to v](-p_{21})$, c.f.\ \eqref{bb}.
\item Specialize $u$ to $p_{21}$; the resulting bundle is $N_C[2p_{11} \to q](-2p_{21})$, c.f.\ \eqref{one}.
\item 
Specialize $u$ to $q$; the resulting bundle is $N_C[p_{11} \to q][q \to v][2p_{21} \to q](-q)$. \\
Then specialize $v$ to $p_{11}$; the resulting bundle is $N_C[p_{11} \to q][q \to p_{11}][2p_{21} \to q](-q)$, c.f.\ \eqref{three}.

\item
Specialize $u$ to $p_{21}$; the resulting bundle is $N_C[p_{11} \to q][v \to p_{21}](-p_{21})$. \\
Exchanging $v$ and $p_{21}$, this is $N_C[p_{11} \to q][p_{21} \to v](-v)$, c.f.\ \eqref{bb}.

\item
Specialize $v$ to $p_{21}$; the resulting bundle is $N_C[u \to p_{21}][2p_{11} \to q][p_{21} \to q](-p_{21})$. \\
Then specialize $u$ to $p_{11}$; the resulting bundle is $N_C[p_{11} \to q][p_{21} \to q](-p_{11} - p_{21})$, c.f.\ \eqref{aa}.

\item
Specialize $v$ to $p_{11}$; the resulting bundle is $N_C[u \to p_{11}][p_{11} \to q][p_{21} \to q][q \to p_{21}](-p_{11})$. \\
Then specialize $u$ to $q$; the resulting bundle is $N_C[p_{11} \to q][p_{21} \to q](-p_{11} - q)$, c.f.\ \eqref{aa}. \qedhere
\end{enumerate}
\end{proof}

Applying Lemma~\ref{lem:disaster}\eqref{aa}\eqref{cc}\eqref{dd}\eqref{ee}, and using \eqref{two} and \eqref{three} directly,
it remains only to show that the
three bundles \eqref{one}--\eqref{three} are semistable.

Let $Q$ be a quadric containing $C$. In cases \eqref{one} and \eqref{three},
specialize $p_{11}$ to one of the two points guaranteed by Lemma \ref{lem:3pts}
for the point $q \in C$;
in case \eqref{two}, specialize both $p_{11}$ and $p_{21}$ to the two points
guaranteed by Lemma \ref{lem:3pts}
for the point $q \in C$.
After these specializations, the inclusion $C \subset Q$ induces
normal bundle exact sequences for the modified bundles \eqref{one}, \eqref{two}, and \eqref{three}:
\begin{gather*}
0 \to N_{C/Q}(-p_{11}) \to N_C[2p_{11} \to q] \to N_Q|_C(-p_{11}) \to 0 \\
0 \to N_{C/Q}(-p_{11} - p_{21}) \to N_C[2p_{11} \to q][2p_{21} \to q] \to N_Q|_C(-p_{11} - p_{21}) \to 0 \\
0 \to N_{C/Q}(-2p_{21}) \to N_C[p_{11} \to q][q \to p_{11}][2p_{21} \to q] \to N_Q|_C(-p_{11} - q) \to 0.
\end{gather*}
These sequences are balanced because $\mu(N_{C/Q}) = 8 = \mu(N_Q|_C)$,
so this establishes the semistability of the modified bundles in \eqref{one}, \eqref{two}, and \eqref{three}
as desired.

\subsection*{\boldmath $(d, g) = (7, 4), (8, 4)$, and $(8, 5)$}
When $(d, g) = (7, 4)$, then $D$ is of degree $5$ and genus $2$.
For uniformity of notation, we write $C = D$.

When $(d, g) = (8, 4)$, then $D$ is of degree $6$ and genus $2$.
We further degenerate $D$ to the union of a general Brill--Noether
curve $C$ of degree $5$ and genus $2$,
with a general $1$-secant line $M$ meeting $C$ at $u$.
Write $v \in M \smallsetminus u$ for another point on $M$.
By Lemma \ref{lem:1_sec_balanced}, in these cases it suffices
to prove semistability of the bundles \ref{cor:limit_gluing_22sec}\eqref{22sec_a}-\eqref{22sec_d}
with the extra modification $[2u \to v]$.

When $(d, g) = (8, 5)$, then $D$ is of degree $6$ and genus $3$.
We further degenerate $D$ to the union of a general Brill--Noether
curve $C$ of degree $5$ and genus $2$,
with a general $2$-secant line $M$ meeting $C$ at $u$ and $v$.
Since $N_{C \cup L}|_L \simeq \O_L(2) \oplus \O_L(2)$ is semistable,
it suffices to show that each of the bundles \eqref{22sec_a}-\eqref{22sec_d} are semistable when restricted to $C$,
i.e.\ it suffices
to prove semistability of the bundles \ref{cor:limit_gluing_22sec}\eqref{22sec_a}-\eqref{22sec_d}
with the extra modification $[u \to v][v \to u]$.

Combining these cases, we have to check the semistability of $9$ modifications of $N_C$.
Applying Lemma~\ref{lem:disaster}\eqref{aa}\eqref{cc}\eqref{dd}\eqref{ee}\eqref{ff}\eqref{gg}\eqref{hh},
and using \eqref{two} and \eqref{three} directly, it suffices to check that the three modifications \eqref{one},
\eqref{two}, and \eqref{three} are semistable for $C$ a general curve of degree $5$ and genus $2$.

Let $Q$ be the unique quadric containing $C$.
In all cases, specialize $p_{21}$ to one of the three points on $C$
guaranteed by Lemma \ref{lem:3pts} for which $N_{C \to q}|_{p_{21}}$ and $N_{C/Q}|_{p_{21}}$ coincide to first order.
Then after these specializations, the inclusion $C \subset Q$ induces the following
normal bundle exact sequences for the modified bundles in \eqref{one}, \eqref{two}, and \eqref{three}:
\begin{gather*}
0 \to N_{C/Q}(-2p_{11}) \to N_C[2p_{11} \to q] \to N_Q|_C \to 0 \\
0 \to N_{C/Q}(-2p_{11} - p_{21}) \to N_C[2p_{11} \to q][2p_{21} \to q] \to N_Q|_C(- p_{21}) \to 0 \\
0 \to N_{C/Q}(-p_{11} - p_{21} - q) \to N_C[p_{11} \to q][q \to p_{11}][2p_{21} \to q] \to N_Q|_C(-p_{21}) \to 0.
\end{gather*}
These sequences are balanced because $\mu(N_{C/Q}) = 12$ and $\mu(N_Q|_C) = 10$,
so this establishes the semistability of the modified bundles in \eqref{one}, \eqref{two}, and \eqref{three}
as desired.

\subsection*{\boldmath The cases $(d, g) = (8, 6)$ and $(9, 6)$}
In these cases, we degenerate to
the union of a general Brill--Noether curve $D$ of degree $d - 3$ and genus $g - 4 = 2$,  
a general $2$-secant line $R_1$, meeting $D$ quasi-transversely precisely at $q$ and $p_{11}$, 
 a general $4$-secant conic $R_2$, meeting $D$ quasi-transversely precisely at $q$, $p_{21}$, $p_{22}$, and $p_{23}$.

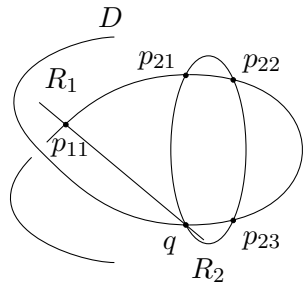
\begin{center}
\begin{tikzpicture}[scale=1]
\draw (1, 3.5) .. controls (0.5, 3.5) and (-1, 3) .. (0, 2);
\draw (0, 2) .. controls (0.5, 1.5) and (1, 1) .. (2, 1);
\draw (0.1, 2.1) .. controls (0.5, 2.5) and (1, 3) .. (2, 3);
\draw (1, 0.5) .. controls (0.5, 0.5) and (-1, 1) .. (-0.1, 1.9);
\draw (2, 1) .. controls (4, 1) and (4, 3) .. (2, 3);
\draw (2.25, 2) ellipse (0.5 and 1.25);
\draw (0,2.63)--(2.19,.8);
\draw (0.3,2.63) node[above]{$R_1$};
\filldraw (0.355,2.335) circle[radius=0.03];
\draw (0.42,2.28) node[below]{$p_{11}$};
\filldraw (1.95,2.99) circle[radius=0.03];
\filldraw (1.95,1) circle[radius=0.03];
\filldraw (2.58,2.93) circle[radius=0.03];
\filldraw (2.58,1.06) circle[radius=0.03];
\draw (1.965,2.93) node[above left]{$p_{21}$};
\draw (1.965,.98) node[below left]{$q$};
\draw (2.575,2.89) node[above right]{$p_{22}$};
\draw (2.573,1.05) node[below right]{$p_{23}$};
\draw (1.25,3.5) node[above left]{$D$};
\draw (2.25,.4) node{$R_2$};
\end{tikzpicture}
\end{center}

Then $R_1$ and $R_2$ satisfy Assumption \ref{assump:balanced_slope}, and so by Lemma \ref{lem:deform_delta}, the union $D \cup R_1 \cup R_2$ deforms to the union of $D$,
a $2$-secant line, and a $4$-secant conic,
which by Lemma \ref{lem:rightcomp}\eqref{lem:rightcomp_2sec} and \eqref{lem:rightcomp_4sec}
is a Brill--Noether curve of degree $d$ and genus $g$.
By Corollary \ref{cor:limit_gluing_4sec2sec}, it suffices to check that the two bundles
\begin{enumerate}[(a)]
\item $N_D[p_{21} \to p_{21}'][p_{22} \to p_{22}'][p_{23} \to p_{23}'][2p_{11} \to q]$ and
\item $N_D[p_{21} \to p_{21}'][p_{22} \to p_{22}'][p_{23} \to p_{23}'][p_{11} \to q][q \to p_{11}]$
\end{enumerate}
are stable when $D$ is a general curve of degree $d - 3$ and genus $2$.
Limiting $p_{11}$ to $p_{21}$, these bundles fit into families whose central fibers are
\begin{enumerate}[(a)]
\item $N_D[p_{22} \to p_{22}'][p_{23} \to p_{23}'][p_{21} \to q]$
\item $N_D[p_{22} \to p_{22}'][p_{23} \to p_{23}'][q \to p_{21}]$
\end{enumerate}
These bundles are symmetric under exchanging $p_{21}$ and $q$,
so it suffices to show the stability of the first bundle.

When $(d, g) = (8, 6)$, then $D$ is of degree $5$ and genus $2$;
in this case, for uniformity of notation, we write $C = D$,
so our problem is simply to show the stability of the bundle
\begin{equation} \label{g6}
N_C[p_{22} \to p_{22}'][p_{23} \to p_{23}'][p_{21} \to q].
\end{equation}

When $(d, g) = (9, 6)$, then $D$ is of degree $6$ and genus $2$.
We further degenerate $D$ to the union of a general Brill--Noether
curve $C$ of degree $5$ and genus $2$,
with a general $1$-secant line $M$ meeting $C$ at $u$.
Write $v \in M \smallsetminus u$ for another point on $M$.
By Lemma \ref{lem:1_sec_balanced}, in these cases it suffices to prove stability for
the bundle
\[N_C[p_{22} \to p_{22}'][p_{23} \to p_{23}'][p_{21} \to q][2u \to v].\]
Limiting $u$ to $p_{21}$ reduces the stability of this bundle to the stability of
\[N_D[p_{22} \to p_{22}'][p_{23} \to p_{23}'][p_{21} \to v],\]
and subsequently limiting $v$ to $q$ reduces
its stability to the stability of \eqref{g6}.

All that remains is thus to show that \eqref{g6} is stable.
The normal bundle exact sequence for the inclusion of $C$ in the unique quadric $Q$ containing it gives rise to the exact sequence
\begin{equation}\label{86_quadric} 0 \to N_{C/Q}(-p_{21}-p_{22} - p_{23}) \to  {N_C[p_{22} \to p_{22}'][p_{23} \to p_{23}'][p_{21} \to q] } \to \O_C(2) \to 0. \end{equation}
These bundles have slopes $9$, $9.5$, and $10$, respectively; hence it suffices to show that this sequence is nonsplit, i.e.\ that
\[H^0(N_C(-2)[p_{22} \to p_{22}'][p_{23} \to p_{23}'][p_{21} \to q]) = 0.\]

By Lemma \ref{lem:52_sections_from_quadric}, all sections of $N_C(-2)$ come from $H^0(N_{C/Q}(-2))$, which has dimension $2$.  After imposing three negative modifications out of the quadric at general points,
we therefore have no global sections as desired.

\section{ Curves of degree $6$ and genus $2$}

This case was done by Sacchiero in \cite{sacchiero3}.
For completeness, we provide a characteristic-independent proof here.
We shall need the following lemma:

\begin{lem}\label{lem:exact_seq}
Let $E$ be a vector bundle on a smooth curve $C$ sitting in an exact sequence
\[0 \to L_1 \to E \to L_2 \to 0,\]
where $L_1$ and $L_2$ are line bundles.
 If 
$\mu(L_2) = \mu(L_1) + 2$, and 
\[\Hom(L_2(-p), E ) \simeq H^0(E \otimes L_2^\vee(p)) = 0\] 
for all $p \in C$, then $E$ is stable.
\end{lem}
\begin{proof}
Let $\phi \colon F \hookrightarrow E$ be a line subbundle
(which recall is always assumed to be saturated).
Then either $\phi$ factors through $L_1 \hookrightarrow E$, in which case
$F \simeq L_1$ is not destabilizing, or
projection from $E$ to $L_2$ gives a nonzero map $F \to L_2$.

In the second case, 
$F \simeq L_2 (-p_1 - \cdots - p_n)$.
Since $\Hom(L_2(-p), E) = 0$ for all $p \in C$ by assumption,
but $\Hom(L_2 (-p_1 - \cdots - p_n), E) \neq 0$,
we must have $n \geq 2$. Therefore
\[\mu(F) = \mu(L_2) - n = \mu(E) - n + 1 < \mu(E). \qedhere\]
\end{proof}

Now let $C$ be a general Brill--Noether curve of degree $d = 6$ and genus $g = 2$.
Since $d > g + r$, our curve $C$ is a projection of a general Brill--Noether curve $\tilde{C} \subset \pp^4$;
by Lemma 13.2 and the proof of Proposition 13.5 of \cite{aly}, $\tilde{C}$
is a quadric section of a cubic scroll.
Thus, $C$ lies on
a cubic surface $S$ singular along a line (the projection of the cubic scroll),
and the normal bundle exact sequence for $C$ in $S$ gives
\begin{equation} \label{inS}
0 \to \O_C(2) \to N_{C/\pp^3} \to L \to 0,
\end{equation}
for some line bundle \(L\). Taking the second wedge power, we have
\[\O_C(2) \otimes L \simeq \wedge^2 N_{C/\pp^3} = K_C(4).\]
Thus \(L \simeq K_C(2)\).
We have $\mu(\O_C(2)) =  12$ and $\mu(K_C(2)) = 14$,
so by Lemma \ref{lem:exact_seq}, it suffices to show for any $p \in C$, 
\[H^0(N_C(-2) \otimes K_C^\vee(p)) = 0.\]
Let $q \in C$ be conjugate to $p$ under the hyperelliptic involution on $C$,
so $K_C^\vee(p) \simeq \O_C(-q)$ and we must show
$H^0(N_C(-2)(-q)) = 0$.
As $N_{C/S}(-2) \simeq \O_C$ has one nowhere-vanishing section, it suffices to show $N_{C/S}(-2) \hookrightarrow N_C(-2)$ is surjective on global sections; i.e., that $h^0(N_C(-2)) = 1$.

We now prove this by degeneration. (We could not degenerate first, since
our desired degeneration would break the exact sequence \eqref{inS}.) 
Namely, we degenerate $C$ to the union $D \cup_u L$ of a general curve $D$ of degree $5$ and genus $2$, and a general $1$-secant line $L$ meeting at the point $u$.  Let $v$ be a point on $L$ away from $u$. 
By \cite[Lemma 8.5]{aly}, it suffices to show $h^0(N_D(-2)(u)[2u \to v]) = 1$.

Let $Q$ be the unique quadric containing $D$.
By Lemma \ref{lem:52_sections_from_quadric}, $H^0(N_D(-2))$ is $2$-dimensional.
When we twist up by $u$, we have an exact sequence
\[0 \to N_{D/Q}(-2)(u) \to N_D(-2)(u) \to \O_D(u) \to 0.\]
As $N_{D/Q}(-2)(u) \simeq K_D(u)$ has exactly a $2$-dimensional space of global sections and vanishing $H^1$,
the associated long exact sequence in cohomology
gives $h^0(N_D(-2)(u)) = 3$.
Consequently, the image of the evaluation map
\[H^0(N_D(-2)(u)) \to N_D(-2)(u)|_u\]
is a $1$-dimensional subspace of the fiber at $u$.
Since the line $L$ is general, the fiber $N_{D \to v}|_u$ will not coincide with this $1$-dimensional subspace.
Therefore, the inclusion $N_D(-2) \subset N_D(-2)(u)[u \to v]$
induces an isomorphism on global sections.
Combining this with Lemma~\ref{lem:52_sections_from_quadric},
the inclusion
\[N_{D/Q}(-2) \subset N_D(-2)(u)[u \to v]\]
also induces an isomorphism on global sections.
Modifying once more towards $v$, and noting that the generality of $v$
guarantees that $N_{D \to v}$ and $N_{D/Q}$ are transverse at $u$,
we conclude that
$N_{D/Q}(-2)(-u) \subset N_D(-2)(u)[2u \to v]$
induces an isomorphism on global sections. Thus
\[h^0(N_D(-2)(u)[2u \to v]) = h^0(N_{D/Q}(-2)(-u)) = h^0(K_D(-u)) = 1.\]

\section{Curves of degree $7$ and genus $5$}

In this section, for completeness we recall Ballico and Ellia's argument \cite{ballicoellia} that shows that if $C$ is a non-hyperelliptic and non-trigonal space curve of degree 7 and genus 5, then $N_C$ is stable. Equivalently, they show that $N_C^\vee(3)$ is stable.
The bundle $N_C^\vee(3)$ has degree 6, hence we need that it does not admit a line bundle of degree 3 or more. Let
\[0 \to L \to N_C^\vee(3) \to M \to 0\]
be a destabilizing sequence. An elementary Riemann-Roch calculation shows that $h^0(\mathcal{I}_C (3)) \geq 3$,
where $\mathcal{I}_C$ denotes the ideal sheaf of $C$ in $\pp^3$.
Since there cannot be a cubic surface double along a curve of degree $7$, the long exact sequence associated to the exact sequence
\[0 \to \mathcal{I}_C^2 (3) \to \mathcal{I}_C (3)  \to N_C^\vee(3) \to  0\]
shows that the image of
\[h\colon H^0(\mathcal{I}_C(3)) \to H^0(N_C^\vee(3))\]
has dimension at least $3$.
Consequently,
\[\dim(H^0(L) \cap \mathrm{im}(h)) + \dim(H^0(M)) \geq 3.\]

If the degree of $L$ is at least $3$, then the degree of $M$ is at most $3$.
Since the curve is not trigonal or hyperelliptic, we conclude that $h^0(M) \leq 1$. Hence, $\dim(H^0(L) \cap \mathrm{im}(h)) \geq 2$. Thus, there are two cubics in the ideal of $C$ whose image in  $N_C^\vee(3)$ lie in the same line subbundle $L$. Hence, these cubics are everywhere tangent along $C$. By Bezout's Theorem, these cubic surfaces intersect in a curve of degree $9$ and cannot be tangent along a curve of degree $7$. Consequently, $N_C^\vee(3)$ cannot have a line subbundle of degree $3$ or more and is stable.

This completes the proof of Theorem~\ref{thm:main}.

\bibliographystyle{plain}

\end{document}